\documentclass[preprint,3p]{elsarticle}

\usepackage{bm,stmaryrd,amsmath,amssymb}
\numberwithin{equation}{section}
\numberwithin{figure}{section}
\usepackage{tabularx,epsfig,color,tabularx,epstopdf}
\usepackage{epstopdf,caption,graphicx,subcaption,mathtools}
\usepackage{algorithm,algorithmic,multirow}
\usepackage{amsfonts,amsthm}
\usepackage{booktabs}
\usepackage{appendix,color}
\usepackage[colorlinks]{hyperref}
\usepackage{appendix}
\usepackage{float}

\def\<{\left\langle}
\def\>{\right\rangle}

\newtheorem{exam}{Example}

\newcommand{\be}{\begin{equation}}
\newcommand{\ee}{\end{equation}}
\newcommand{\bes}{\begin{equation*}}
\newcommand{\ees}{\end{equation*}}
\newcommand{\ba}{\begin{array}}
\newcommand{\ea}{\end{array}}
\newcommand{\bea}{\begin{eqnarray}}
\newcommand{\eea}{\end{eqnarray}}
\newcommand{\beas}{\begin{eqnarray*}}
\newcommand{\eeas}{\end{eqnarray*}}

\newcommand{\fl}[2]{\frac{#1}{#2}}

\newcommand{\bx}{{\bm x}}

\newcommand{\im}{{\mathrm i}}

\newcommand{\cA}{{\mathcal A}}
\newcommand{\cB}{{\mathcal B}}
\newcommand{\cC}{{\mathcal C}}
\newcommand{\cD}{{\mathcal D}}

\newcommand{\bu}{{\bm u}}
\newcommand{\bv}{{\bm v}}
\newcommand{\bbf}{{\bm f}}

\newcommand{\bbg}{{\bm g}}

\newcommand{\bbs}{{\mathbf s}}
\newcommand{\bss}{{\mathbf S}}

\newcommand{\cH}{{\mathcal H}}

\newcommand{\sH}{{\mathsf H}}
\newcommand{\og}{{\omega}}
\newcommand{\diag}{{\rm diag}}
\newcommand{\supp}{{\rm supp}}
\newcommand{\spa}{{\rm span}}

\newcommand{\cI}{{\mathcal I_N}}
\newcommand{\ns}{{\rm null}}

\newcommand{\lf}{\left(}
\newcommand{\rg}{\right)}

\newcommand{\dif}{{\mathrm d}}

\newcommand{\p}{\partial}
\newcommand{\gm}{{\gamma}}
\newcommand{\ap}{{\alpha}}

\newcommand{\lag}{\left\langle}
\newcommand{\rag}{\right\rangle}

\theoremstyle{plain}

\newtheorem{remark}{Remark}[section]

\newtheorem{thm}{Theorem}[section]
\newtheorem{lem}{Lemma}[section]
\newtheorem{corollary}{Corollary}[section]

\newtheorem{lemma}{Lemma A.\hspace{-0.2em}}

\begin{document}

\title{An efficient Fourier spectral algorithm for the Bogoliubov-de Gennes excitation eigenvalue problem}


\author[tjfu]{Yu LI}
\ead{liyu@tjufe.edu.cn}
\author[scu]{Zhixuan LI}
\ead{zhixuan\_li@stu.scu.edu.cn}
\author[smtju]{Manting XIE\corref{5}}
\ead{mtxie@tju.edu.cn}
\author[tju]{Yong ZHANG}
\ead{Zhang\_Yong@tju.edu.cn}

\address[tjfu]{Coordinated Innovation Center for Computable Modeling in Management Science,\\
 Tianjin University of Finance and Economics, Tianjin 300222, China}
\address[scu]{School of Mathematics, Sichuan University, Chengdu 610064, China}
\address[smtju]{School of Mathematics and KL-AAGDM, Tianjin University, Tianjin, 300350, China}
\address[tju]{Center for Applied Mathematics and KL-AAGDM, Tianjin University, Tianjin 300072, China}

\cortext[5]{Corresponding author}

\begin{abstract}

In this paper, we propose an efficient Fourier spectral algorithm for an eigenvalue problem, that is, the Bogoliubov-de Gennes (BdG) equation
arsing from spin-1 Bose-Einstein condensates (BEC) to describe the elementary/collective excitations around the mean-field ground state.
The BdG equation is essentially a constrained eigenvalue/eigenfunction system.
Firstly, we investigate its analytical properties, including exact eigenpairs,
generalized nullspace, and bi-orthogonality of eigenspaces.
Secondly, by combining the standard Fourier spectral method for spatial discretization and a stable Gram-Schmidt bi-orthogonal algorithm,
we develop a subspace iterative solver for such a large-scale dense eigenvalue problem,
and it proves to be numerically stable, efficient, and accurate.
Our solver is matrix-free and the operator-function evaluation is accelerated by discrete Fast Fourier Transform (FFT) with almost optimal efficiency.
Therefore, it is memory-friendly and efficient for large-scale problems.
Furthermore, we give a rigorous and detailed numerical analysis on the stability and spectral convergence.
Finally, we present extensive numerical results to illustrate the spectral accuracy and efficiency, and
investigate the excitation spectrum and Bogoliubov amplitudes around the ground state in $1$--$3$ spatial dimensions.

\end{abstract}


\maketitle


\section{Introduction}
Since the experiment realizations of degenerate spinor BEC with ultracold rubidium \cite{Hall1998Dynamics} and sodium \cite{Stamper-Kurn1998Optical},
where the internal degrees of freedom of spin were activated by the optical dipole potential trap, there followed a vast of
physical experiments of various spinor BEC  \cite{Matthews1998Dynamical,Stenger1998Spin} and theoretically investigations to
explore various peculiar  quantum phenomena \cite{Alkai-Japan}.
In spinor BEC, particles of different hyperfine states allow for different angular momentum spaces, leading to very rich spin structures.
Therefore, degenerate spinor quantum gases simultaneously exhibit magnetic and superfluid properties and shed light on a wide range of topics such as
 topological quantum structures, fractional quantum Hall effects, etc \cite{Kawaguchi2012Spinor,Stamper-Kurn2013Spinor}.

At temperature $T$ much smaller than the critical condensate temperature $T_c$ \cite{Bao2017Mathematical},
a spin-1 BEC is well described by the three-component wave function $\Psi(\bx, t) = (\psi_1(\bx, t),
\psi_0(\bx, t), \psi_{-1}(\bx, t))^\top$, whose evolution is governed by the coupled Gross-Pitaevskii
equations (GPEs) \cite{Ho1998Spinor,Stamper-Kurn2013Spinor,Zhang2003Mean}.
In dimensionless form, the $d$-dimensional ($d=1$, $2$, or $3$) GPEs could be unified as
\cite{Bao2017Mathematical,Bao2007mass}
\begin{subequations}\label{GPE}
\begin{align}
\im\partial_t \psi_1&=\left[-\frac{1}{2} \nabla^2  +V(\bx)+ \beta_n\rho + \beta_s\left(|\psi_0|^2+\rho_z\right)\right]
\psi_1 + \beta_s\psi_0^2\bar\psi_{-1},\label{spin1}\\
\im\partial_t \psi_0&= \left[-\frac{1}{2} \nabla^2  +V(\bx)+ \beta_n\rho +\beta_s\left(|\psi_1|^2+|\psi_{-1}|^2\right)\right]
\psi_0 + 2\beta_s\psi_{-1}\bar\psi_{0}\psi_1,\label{spin0}\\
\im\partial_t \psi_{-1}&= \left[-\frac{1}{2} \nabla^2  +V(\bx)+ \beta_n\rho
+\beta_s\left(|\psi_0|^2-\rho_z\right)\right]
\psi_{-1} + \beta_s\bar\psi_{1}\psi_0^2,\label{spin-1}
\end{align}
\end{subequations}
where $t$ denotes time and ${\bx}=x \in {\mathbb R}$, ${\bx}=(x, y)^\top \in {\mathbb R}^2$
and/or ${\bx}=(x, y, z)^\top \in {\mathbb R}^3$ is the Cartesian coordinate vector,
 $\rho= \sum_{j=-1}^1|\psi_j|^2$ is the density function and $\rho_z = |\psi_1|^2-|\psi_{-1}|^2$.
$V(\bx)$ is a real-valued  external potential that is case-dependent
and one common choice is the following harmonic trapping potential
\be\label{Vpoten}
V(\bx) = \frac{1}{2}\left\{\begin{array}{ll}
 \gm_x^2x^2, & d = 1,\\[0.3em]
 \gm_x^2x^2 + \gm_y^2y^2, & d = 2,\\[0.3em]
\gm_x^2x^2 + \gm_y^2y^2 + \gm_z^2z^2, \ \ &d = 3.
\end{array}\right.
\ee
Here, $\gm_\alpha>0\ (\alpha=x,y,z)$ are dimensionless constants proportional to the trapping frequency in the $\alpha$-direction.
Constant $\beta_n$ denotes the mean-field interaction, and $\beta_s$ represents spin-exchange interaction
with positive/negative $\beta_s$ corresponding to the anti-ferromagnetic/ferromagnetic case.

Introduce the triple of spin-$1$ Pauli matrices $\bss=(\bss_x,\bss_y,\bss_z)$ as
\bes
\bss_x = \frac{1}{\sqrt{2}}\begin{bmatrix}
0 & 1 & 0 \\
1 & 0 & 1 \\
0 & 1 & 0
\end{bmatrix},\ \ \
\bss_y = \frac{\im}{\sqrt{2}}\begin{bmatrix}
0 & -1 & 0 \\
1 & 0 & -1 \\
0 & 1 & 0
\end{bmatrix},\ \ \
\bss_z = \begin{bmatrix}
1 & 0 & 0 \\
0 & 0 & 0 \\
0 & 0 & -1
\end{bmatrix},
\ees
and the spin vector
$$\bbs(\Psi)=(\bbs_x(\Psi),\bbs_y(\Psi),\bbs_z(\Psi))^\top = (\Psi^\sH\bss_x\Psi,\Psi^\sH\bss_y\Psi,\Psi^\sH\bss_z\Psi)^\top,$$
where $\bar{\bm\zeta},~\bm\zeta^\mathsf{H}:= (\bar {\bm\zeta})^\top$ are conjugate and conjugate transpose of vector $\bm\zeta\in {\mathbb C}^{3\times 1}$ respectively.
Note that Pauli matrices are all Hermitian, i.e., $\bss_j^\sH= \bss_j$ for $j = x,y,z$.
The GPEs \eqref{GPE} can be written in a compact form
\be\label{GPEs}
\im\partial_t\Psi = \left(\big(L+\beta_n\rho\big)\mathbf{I}_3 + \beta_s\,\bss\cdot\bbs(\Psi)\right) \Psi,
\ee
where $L = -\frac{1}{2} \nabla^2  +V(\bx)$, $\mathbf{I}_3$ is identity matrix, and
\beas
\bss\cdot\bbs(\Psi) &:=& \bss_x\,\bbs_x(\Psi)+\bss_y\,\bbs_y(\Psi)+\bss_z\,\bbs_z(\Psi).
\eeas

The GPEs \eqref{GPE} conserves three important quantities:
the \textsl{mass}
\beas \label{constrain-C}
\mathcal N(\Psi(\cdot,t))=\|\Psi(\cdot,t)\|^{2}:=\int_{\mathbb{R}^{d}}\sum_{j=-1}^1|\psi_j(\bx)|^{2}\,\dif\bx\equiv \mathcal N(\Psi(\cdot,0)) = 1,\qquad t\ge0,
\eeas
the \textsl{magnetization} (with $-1\leq M\leq1$)
\beas \label{constrain-M}
\mathcal M(\Psi(\cdot,t)):=\int_{\mathbb{R}^{d}}\left(|\psi_1(\bx)|^{2}-|\psi_{-1}(\bx)|^{2}\right)\dif\bx\equiv \mathcal M(\Psi(\cdot,0)) = M,\qquad t\ge0,
\eeas
and \textsl{energy}
\beas\label{energy}\quad
\mathcal E(\Psi(\cdot,t))&:=&\int_{{\mathbb R}^d} \bigg\{\sum_{j=-1}^1\Big(\frac{1}{2}\left|\nabla\psi_j(\bx)
\right|^2 +V(\bx)|\psi_j(\bx)|^2\Big) +  \frac{\beta_n}{2}|\Psi(\bx)|^4 + \frac{\beta_s}{2}|\bbs(\Psi(\bx))|^2\bigg\}\,\dif \bx\nonumber\\
&\equiv& \mathcal E(\Psi(\cdot,0)),\quad \ t\geq0,
\eeas
where $|\Psi|$ is length of vector $\Psi$, that is, $ |\Psi|:= \sqrt{|\psi_1|^2 + |\psi_0|^2 + |\psi_{-1}|^2 } = \sqrt{\rho}$.
The ground state, denoted by $\Phi_g = (\phi_1^g,\phi_0^g,\phi_{-1}^g)^\top$, is defined as minimizer of the following problem
\be\label{groundDef}
\Phi_g =\arg\min_{\Phi\in \mathcal S_M} \; \mathcal{E}(\Phi),
\ee
where $\mathcal{S}_M$ is the constraint functional space
$$\mathcal{S}_M:=\Big\{\Phi= (\phi_1,\phi_0,\phi_{-1})^\top \ \big|\
	\|\Phi\|^2=1,~\|\phi_1\|^2-\|\phi_{-1}\|^2=M,~ \mathcal E(\Phi)<\infty\Big\}.$$
The ground state $\Phi_g$ satisfies the following Euler-Lagrange equations
\be\label{NLEP}
\Lambda\Phi = \mathbf{H} \Phi: = \Big(\big(L + \beta_n\rho\big)\mathbf{I}_3 + \beta_s\, \bss \cdot \bbs(\Phi)\Big) \Phi ,
\ee
where $\mathbf{H} = \diag(H_1,H_0,H_{-1})$ is the Hamiltonian operators,
$\Phi = (\phi_1,\phi_0,\phi_{-1})^{\top}$,
$\Lambda = \diag(\mu_1,\mu_0,\mu_{-1})$ with entry $\mu_j$, given below as the chemical potential associated with the $j$-th component \cite{Bao2017Mathematical,Bao2007mass},
\begin{equation}
\mu_j = \dfrac{\int_{\mathbb{R}^d}\overline{\phi_j} H_j\phi_j\dif\bx}{\int_{\mathbb{R}^d}|\phi_j|^2\dif\bx},
\quad j = 1,0,-1,
\label{chem-formula}
\end{equation}
satisfying relation $\mu_1+\mu_{-1} =2\mu_0$.
We can construct stationary states of \eqref{GPE} in the following way
\bea\Psi_s(\bx,t) =\diag(e^{-\im \mu_1t},e^{-\im \mu_0t},e^{-\im \mu_{-1}t}) \Phi_s(\bx),\eea
where $\Phi_s=(\phi^s_{1},\phi^s_0,\phi^s_{-1})^\top$ and $\left\{\mu_{1}, \mu_0,\mu_{-1}\right\}$ are solutions to the nonlinear eigenvalue problem \eqref{NLEP}.

\

On the mean-field level, the many-body effects are absent and the GPE \eqref{GPE} proves to be valid
for the spin-1 BEC \cite{Ho1998Spinor,Stamper-Kurn2013Spinor,Zhang2003Mean}.
However, due to the many-body effect of interatomic interactions, there are excitations in the system even in the lowest energy state, which could be regarded as quasi-particles and are known as elementary/collective excitations \cite{Bloch2008Many}.
We first need to go beyond the mean-field theory to study elementary/collective excitations.
Under the proper assumption, the elementary/collective excitation around the mean-field stationary state could be well described within the Bogoliubov theory,
which resulted in the celebrated Bogoliubov-de Gennes (BdG) equations \cite{Baillie2017Collective,Deng2020Spin,Gao2020Numerical,Huhtamaeki2011Elementary}.
As elementary/collective excitation of BEC provides fundamental information about the ultracold quantum state,
different phases of a trapped spin-1 BEC feature distinctive elementary/collective excitations, therefore,
they could be used to distinguish various phases.

\medskip

The first experimental measurements of the lowest elementary/collective modes of BECs were performed in \cite{Jin1996Collective}. Since then,
great enthusiasm has been stimulated to mathematical and numerical investigation of BdG equations in the past few decades \cite{Hu2004Analytical,Stringari1996Collective}.
Along the numerical front, there are many successful researches devoted to BdG corresponding to various BECs based on the well-known ARPACK library,
and we refer to \cite{Edwards1996Collective,Danaila2016Vector,Tang2022Spectrally,Deng2020Spin} for an incomplete list.
The spatial discretization of the wavefunction, eigenfunction, and linear differential operators can be categorized into finite difference method \cite{Edwards1996Collective,Gao2020Numerical}, finite-element method \cite{BdG-FreeFEM,Danaila2016Vector}
 and Fourier/sine spectral method \cite{Tang2022Spectrally,Zhang2021}.
For low dimension problem ($d=1,2$), the resulting eigensystem is solved with either direct eigensolver,
e.g., using ``eigs'' with MATLAB
or iterative subspace solver, such as ARPACK
or the locally optimal block preconditioned 4-d conjugate gradient \cite{Bai2012Minimization,Bai2013Minimization}.

Up to now, there have been relatively few theoretical studies on the excitation of spin-1 BEC.
Most numerical studies have focused on lower-dimensional cases ($d=1$ or $2$),
for example, the low-lying elementary/collective excitations of quasi-two-dimensional rotating systems  \cite{Deng2020Spin,Zhang2021}
and cylindrically symmetric systems \cite{Wilson2009Stability}.
Chen \textsl{et al.} \cite{Chen2017Collective} directly solved the BdG equations for spin-orbit BEC using the Arnoldi method,
while the Arnoldi method with ARPACK was adapted to study the dipolar spin-1 BEC \cite{Deng2020Spin}.
Recently, Tang \textsl{et al.} \cite{Tang2022Spectrally,Zhang2021} construct efficient
and spectral numerical algorithms for the BdG equations of dipolar BEC in higher-dimensional spaces.

As far as we know, there have been quite few mathematical or numerical studies on the BdG excitations of spin-1 BEC.
Hence, it is necessary to develop mathematical theories and construct an accurate and efficient eigensolver.
To numerically study the elementary/collective excitations, there are two main challenges:
(i) an accurate stationary state solver, especially the ground state solver;
(ii) an accurate and efficient BdG solver.
As for the computation of the ground state, various numerical methods have been proposed for the spin-1 BEC,
such as the gradient flow method \cite{Bao2017Mathematical,Bao2007mass},
the gradient flow with the Lagrange multiplier (GFLM) method \cite{Liu2021Normalized}, 
 regularized Newton method \cite{TianCai} and efficiently preconditioned conjugated gradient method (PCG) \cite{ATZ-CiCP-PCG} etc.
Here, we adopt PCG to compute the ground state due to its simplicity and effectiveness.

As is known, the structure of eigenvalue distribution and the generalized nullspace is of essential importance
to eigensolver performance in terms of efficiency, accuracy and stability.
It is noted that eigenvalues, coming in negative/positive pairs, are real and span over the real line to infinity,
and the eigenspaces corresponding to eigenvalues of different magnitudes are bi-orthogonal \cite{LiWangZhang}.
More importantly, for eigenvalue zero, the algebraic multiplicity may not be equal to its geometric multiplicity,
thus leading to a rich and complex generalized nullspace structure.
However, such peculiar inherent structure and analytical properties have not been considered or taken into account in the state-of-art popular solvers,
for example, the ARPACK library is much more a general-purpose eigensolver.
Direct application of the aforementioned solvers may encounter slow convergence or even divergence,
especially at eigenvalue zero, and the performance will definitely deteriorate in high-dimension problems.

To make it more challenging, the eigenfunctions are usually discretized by the Fourier spectral method, therefore, the resulted discrete eigensystem
is a fully-populated dense matrix, which renders prohibitively huge memory costs if one chooses to store the matrix explicitly.
Therefore, it is imperative to adopt some iterative algorithm, where one only needs to provide a matrix-vector product or operator-function evaluation,
so to bypass the huge explicit matrix storage, and gain good parallel scalability for large-scale high dimension problems.

\

Overall, the main objectives of this paper are fourfold:
\begin{itemize}
\item[1.] derive the BdG equations for spin-1 Bose-Einstein condensates around the ground state
and investigate its mathematical structures, including analytical eigenvalues/eigenfunctions,
generalized nullspace,
and eigenspaces' bi-orthogonality;

\item[2.] develop a bi-orthogonal structure-preserving Fourier spectral eigensolver, and
    propose an efficient evaluation of the matrix-vector product based on the Fourier spectral method;

\item[3.] prove the numerical stability and spectral convergence;

\item[4.] verify its spectral accuracy, efficiency, and scalability, and numerically study the excitation spectrum and
Bogoliubov amplitudes around the ground state with different parameters in 1D, 2D, and 3D.

\end{itemize}

The rest of the paper is organized as follows:
In Section \ref{sec:BdG_Prop}, we introduce the BdG equations and derive some analytical properties.
In Section \ref{Numer_alg}, we present details of the Fourier spectral method for space discretization
and propose an efficient eigensolver.
Extensive numerical examples are shown in Section \ref{NumResult} to confirm the performance of our method,
together with some applications to study the solutions to the BdG equations with different parameters in 1D, 2D, and 3D.
Finally, conclusions are drawn in Section \ref{sec:Conclusion}.

\section{The BdG equations and its properties}
\label{sec:BdG_Prop}

\subsection{The Bogoliubov-de Gennes equations}
\label{sec:BdG equation}
To characterize the elementary/collective excitations of a spin-1 BEC,
the Bogoliubov theory \cite{Baillie2017Collective,Deng2020Spin}
begins with the ground state $\Phi_g$ of the CGPEs \eqref{GPEs},
which are  also the solution of the nonlinear eigenvalue problems \eqref{NLEP} with corresponding chemical potentials $\mu_1$, $\mu_0$ and $\mu_{-1}$,
and assumes the evolution of CGPEs \eqref{GPEs} is around $\Phi_g$.
The corresponding wave function $\Psi$ takes the following form \cite{YiLowLying18,Tang2022Spectrally}
\begin{align}\label{waveAssump} \quad
\Psi(\bx,t) =\diag\big(e^{-\im \mu_1t},e^{-\im \mu_0t},e^{-\im \mu_{-1}t}\big)
\left[ \Phi_g(\bx) + \varepsilon \sum_{\ell=1}^\infty \Big( \bu^\ell(\bx) e^{-\im \omega_\ell t} + \bar\bv^\ell (\bx) e^{\im \omega_\ell t}\Big)\right],
\ \bx\in\mathbb{R}^d,\ t > 0.
\end{align}
Here, $0<\varepsilon \ll 1$ is a small quantity used to control the population of quasiparticle excitation, $\omega_\ell\in \mathbb{C}$
is the frequency of the excitations to be determined,
and $\bu^\ell, \bv^\ell \in H^1(\mathbb R^3;\mathbb C^3)$ are the Bogoliubov excitation modes
satisfying normalization condition
\bea
\label{constrain_org}
\int_{\mathbb R^d} \left( | \bu^\ell(\bx)|^2-  | \bv^\ell(\bx)|^2 \right) \dif \bx =
\int_{\mathbb R^d} \sum_{j=-1}^1\left( | u^\ell_j(\bx)|^2-  | v^\ell_j(\bx)|^2 \right) \dif \bx =  1, \quad \ell\in \mathbb Z^{+}.
\eea
where $\bu^\ell = (u_1^\ell,u_0^\ell,u_{-1}^\ell)^\top,\bv^\ell =(v_1^\ell,v_0^\ell,v_{-1}^\ell)^\top$.
Plugging \eqref{waveAssump} into \eqref{GPEs}, by collecting linear terms in $\varepsilon$ and separating frequency $e^{-i\omega_\ell t}$ and $e^{i\omega_\ell t}$,
we obtain the BdG equations as follows
\begin{gather}
\label{BdG_uv}
\begin{bmatrix}
\cA & \cB \\
\cC  & \cD
\end{bmatrix}
\begin{bmatrix}
\bu \\
 \bv
\end{bmatrix}
=\omega\begin{bmatrix}
\bu\\
\bv
\end{bmatrix},
\end{gather}
with constraint
\be
\label{constrain}
\int_{\mathbb{R}^d} \left(|\bu(\bx)|^2-|\bv(\bx)|^2\right)\,\dif\bx=1,
\ee
where all the subscripts $\ell$ are omitted hereafter for simplicity, $\omega$ is the excitation energy and the operators are given explicitly
\bea
\cA &=& (L\mathbf{I}_3 - \Lambda) + \beta_n\big(\rho\mathbf{I}_3 + \Phi_g\Phi_g^\sH\big) + \beta_s\bigg( \bss\cdot\bbs(\Phi_g) + \sum_{j=x,y,z}\bss_j\Phi_g\Phi_g^\sH\bss_j^\sH\bigg),\label{def_A}\\
\cB &=& \beta_n\Phi_g\Phi_g^\top + \beta_s\sum_{j=x,y,z}\bss_j\Phi_g\Phi_g^\top\bss_j^\top,\label{def_B}\\
\cC &=& -\bigg[\beta_n\bar\Phi_g\Phi_g^\sH + \beta_s\sum_{j=x,y,z}\bar\bss_j\bar\Phi_g\Phi_g^\sH\bss_j^\sH\bigg],\label{def_C}\\
\cD &=& -\bigg[(L\mathbf{I}_3 - \Lambda) + \beta_n\big(\rho\mathbf{I}_3 + \bar\Phi_g\Phi_g^\top\big) + \beta_s\bigg({\bar\bss\cdot\bbs(\Phi_g)} + \sum_{j=x,y,z}\bar\bss_j\bar\Phi_g\Phi_g^\top\bss_j^\top \bigg)\bigg].\label{def_D}
\eea

It is easy to check that $\mathcal A$ and $\mathcal D$ are both Hermitian operators,
i.e., $\cA^* = \cA, ~\cD^*=\cD$, and $\cB^*= -\cC$, where symbol $*$ denotes
the operator adjoint associated with inner product $\left\langle \bbf,\bbg \right\rangle
:=\sum_{j=-1}^1\int_{\mathbb{R}^d} {f_j}(\bx)\overline{g_j(\bx)}\, \dif \bx$.
Furthermore, we have $\overline{\cA\bbg} = -\cD\bar\bbg$,
$\overline{\cB\bbg} = -\cC\bar\bbg$, $\forall\,\bbg\in H^1(\mathbb{R}^d;\mathbb C^3)$,
and it immediately implies the operators's finite-dimension subspace representations,
denoted by matrices $A,B,C$ and $D$, are either Hermitian or symmetric, that is,
\bea
A^\sH = A, \quad D = -\overline A ,  \quad B^\top = B \quad \mbox{and} \quad C  = -\overline B.
\eea
The matrices representation of \eqref{BdG_uv} reads as follows
\begin{gather}
\begin{bmatrix}
A & B \\
-\overline B   &- \overline A
\end{bmatrix}
\begin{bmatrix}
{\mathbf u} \\
{\mathbf v}
\end{bmatrix}
=\omega\begin{bmatrix}
{\mathbf u} \\
{\mathbf v}
\end{bmatrix},
\end{gather}
and it coincides with the Bethe-Salpeter Hamiltonian (BSH) matrix arising from optical absorption spectrum analysis \cite{Shao2018structure}.

\subsection{Analytical properties}
\label{sec:Prop}
In this section, we derive some analytical properties of the BdG equations as well as the structure of generalized nullspace for both
ferromagnetic and anti-ferromagnetic cases. They might serve as benchmarks for numerical solutions or help to design an efficient eigensolver.

\begin{thm}[\textbf{Symmetric distribution}]\label{lem:negtive_eigenvalue}
If $\{\omega; \bu, \bv\}$ is a solution pair to the BdG equations
\eqref{BdG_uv}, then $\{-{\bar\omega};\bar{ \bv}, \bar{ \bu} \}$ is also a solution pair.
Furthermore,  If  $(\bu,\bv)$ satisfies the normalization constraint \eqref{constrain},
i.e., the elementary/collective excitations, the frequency $\omega$ is real.

\end{thm}
\begin{proof} The first conclusion can be proved by taking the conjugate of Eqn. \eqref{BdG_uv}.
Multiplying the first/second equation of \eqref{BdG_uv} by $\bu$/$\bv$ respectively
and integrating each equation concerning $\bx$,
we subtract the second integration from the first one to obtain the following
\begin{equation}\label{real_eg3}
\lag \bu,\cA \bu\rag + \lag \bu,\cB \bv\rag - \lag\bv,\cC \bu\rag
- \lag\bv,\cD \bv\rag = \bar \omega\int_{\mathbb{R}^d}(|\bu(\bx)|^2-|\bv(\bx)|^2)\,\dif \bx.
\end{equation}
From Eqn. \eqref{constrain}, we can see that $\omega$ is real.

\end{proof}

As shown in \cite{Bao2017Mathematical}, under appropriate assumptions,  the ground states can be chosen as real-valued functions.
Therefore, in this article, we shall focus on the real ground state hereafter, and all operators involved are real.
To be more specific,
\bes\label{real-oper}
\cD = -\cA,\quad \cC = -\cB.
\ees
The BdG equation \eqref{BdG_uv} is a linear response eigenvalue problem of the following form
\begin{gather}
\label{BdG_new}
\begin{bmatrix}
\cA & \cB \\
-\cB  & -\cA
\end{bmatrix}
\begin{bmatrix}
\bu\\
\bv
\end{bmatrix}
=\omega\begin{bmatrix}
\bu\\
\bv
\end{bmatrix}.
\end{gather}
By applying a change of variables
\be\label{c-o-v}
\bu=\bbf+\bbg,~\bv=\bbf-\bbg,
\ee
the above equation can be reformulated
\bea \label{HpHm-eq}
\cH\begin{bmatrix}
	\bbf  \\ \bbg
	\end{bmatrix} :=
\begin{bmatrix}
	\mathcal O&\mathcal H_- \\
\mathcal H_+&\mathcal O
\end{bmatrix}
\begin{bmatrix}
	\bbf  \\ \bbg
	\end{bmatrix} =
	\omega
\begin{bmatrix}
	 \bbf  \\ \bbg
\end{bmatrix},
\eea
where $\mathcal H_+:=\cA+\cB$ and $\mathcal H_-:=\cA-\cB$ are both Hermitian operators, and the constraint \eqref{constrain} is reformulated in $(\bbf,\bbg)$ as
\bea\label{constrainfg}
\lag\bbf,{\bbg}\rag=\fl{1}{4}.\eea

\begin{remark}
If $\cH_+$ is positive definite, according to Lemma 2.3 of \cite{LiWangZhang},
the constrain \eqref{constrain} implies that the eigenvalue $\omega$ is nonzero, therefore,
we shall only focus on the {\bf non-zero} eigenvalues and their eigenfunctions.
\end{remark}

\begin{thm}[\textbf{Analytical eigenparis}]\label{lemAnalytical}
   Let $\Phi_g=(\phi_1^g,\phi_0^g,\phi_{-1}^g)^\top$ be real-valued ground states for Eqn. \eqref{GPE} with harmonic trapping potential \eqref{Vpoten},
we have analytical solutions to BdG equation \eqref{BdG_uv} as follows
\bea\label{eq:eig-pair}
\{\og_\alpha;\bu_\alpha,\bv_\alpha\}:=\left\{\gm_\ap ; \fl{1}{\sqrt{2}} \lf \gm_\alpha^{-\fl{1}{2}}
\p_\alpha \Phi_g-\gm_\alpha^{\fl{1}{2}}\alpha\Phi_g \rg, \fl{1}{\sqrt{2}} \lf \gm_\ap^{-\fl{1}{2}}\p_\ap\Phi_g+\gm_\ap^{\fl{1}{2}}\ap\Phi_g \rg \right\},
\eea
with $\alpha = x$ in one dimension, $\alpha = x,y$ in two dimensions and $\alpha = x, y, z$ in three dimensions.
\end{thm}
\begin{proof}
For simplicity, we only prove the $\alpha =x$ case and extensions to other spatial variables are similar.
Differentiate \eqref{NLEP} with respect to $x$ and combine $\mathcal H_+$ definition, we derive
\beas\label{Hpp}
\mathcal H_+(\p_x\Phi_g)=\gm_x(-\gm_xx\Phi_g),
\eeas
with $\p_x\Phi_g=(\p_x\phi_1^g,\p_x\phi_0^g,\p_x\phi_{-1}^g)^\top$.
Apply $\mathcal H_-$ on $-\gm_xx\Phi_g$, we have
\beas\label{Hmm}
\mathcal H_-(-\gm_xx\Phi_g)=\gm_x(\p_x\Phi_g).
\eeas
Therefore, we can see that $(\partial_x\Phi_g^\top, - \gamma_xx\Phi_g^\top)^\top$
solves \eqref{HpHm-eq} with $\omega =\gamma_x$.
With simple calculations, we have
$$\lag -\gm_xx\Phi_g,\p_x\Phi_g\rag
=\int_{\mathbb{R}^2}\sum_{j=-1}^1\big(-\gm_xx\phi_j^g\p_x\phi_j^g\big)\,{\rm d}x{\rm
d}y = \fl{\gm_x}{2}.$$
Therefore, the normalized function
$(\bbf,\bbg) = \left(\fl{1}{\sqrt{2}}\gm_x^{-\fl{1}{2}}\p_x\Phi_g, -\fl{1}{\sqrt{2}}\gm_x^{\fl{1}{2}}x\Phi_g\right)$
solves \eqref{HpHm-eq} with eigenvalue $\omega  = \gamma_x$. Finally,
we derive the following analytic solutions
\beas
\og_x=\gm_x, \quad \bu_x=\fl{1}{\sqrt{2}}\Big(\gm_x^{-\fl{1}{2}}\p_x\Phi_g-\gm_x^{\fl{1}{2}}x\Phi_g\Big), \quad
\bv_x=\fl{1}{\sqrt{2}}\Big(\gm_x^{-\fl{1}{2}}\p_x\Phi_g+\gm_x^{\fl{1}{2}}x\Phi_g\Big).
\eeas
\end{proof}

\begin{thm}[\textbf{Biorthogonality}]\label{lem:orth}
Assume $\{\omega_i; \bbf_i, \bbg_i\}_{i=1}^2$ are eigenpairs of Eqn. \eqref{HpHm-eq}
with eigenvalues of different magnitudes, i.e., $|\omega_1|\neq|\omega_2|$,
the following biorthogonal properties hold true
\begin{equation*}
\langle \bbf_1,\bbg_2 \rangle
=
\langle \bbf_2,\bbg_1 \rangle
=0.
\end{equation*}
\end{thm}

\begin{proof}
Using Eqn. \eqref{HpHm-eq}, we have
\bes
\og_i^2\lag\bbf_i,\bbg_j\rag  = \lag\mathcal H_-\mathcal H_+\bbf_i,\bbg_j\rag = \lag\bbf_i,\mathcal H_+\mathcal H_-\bbg_j\rag =  \og_j^2\lag\bbf_i,\bbg_j\rag,
\ees
which means
\bes
(\og_i^2 - \og_j^2)\lag\bbf_i,\bbg_j\rag = 0.
\ees
Since $|\og_i|\neq|\og_j|$, we obtain $\lag\bbf_i,\bbg_j\rag = 0, \ \  \mbox{for}\ i\neq j$.
\end{proof}

\subsection{Generalized nullspace}
\label{sec:rootspace}
As is known, the algebraic/geometric multiplicity of eigenvalue zero and the structure of its associated generalized nullspace
are of great importance to the eigensolver's performance in terms of convergence, accuracy and efficiency.
The generalized nullspace of $\mathcal H$ is defined as nullspace of $\mathcal H^{p}$ for some positive integer $p$ such that
\begin{equation*}
\ns(\cH^p)=\ns(\cH^{p+1}).
\end{equation*}
For compact operator $\cH$, the integer $p$ is finite and $\ns(\cH^p)=\ns(\cH^{q})$ holds true for any greater integer $q$, i.e, $q \geq p$.
In this subsection, we shall elaborate on the structure of generalized nullspace for ferromagnetic and antiferromagnetic systems, which is of course closely connected with the nullspace of $\mathcal H_-$ and $\mathcal H_+$.

\vspace{0.35cm}

\noindent {\bf Case I: Ferromagnetic system ($\beta_s <0$)}.

\vspace{0.1cm}

\begin{lem}[{Single Mode approximation (SMA) \cite[Theorem 4.2]{Bao2017Mathematical}}]
\label{lemma:ferr}
If $M\in (-1,1)$, $\beta_s <0$ and $\beta_n+\beta_s\geq 0$,
there exists a ground state solution $\Phi_g = {\bm a}\,\phi_g,~{\bm a} = \big(\frac{1+M}{2},(\frac{1-M^2}{2})^{1/2},\frac{1-M}{2}\big)^\top$,
and the chemical potentials are identical, i.e., $\mu_1=\mu_0=\mu_{-1}$.
Here, $\phi_g$ is the unique positive minimizer of the following energy functional
$$E_{\textrm{SMA}}(\phi) = \int_{\mathbb{R}^d}\left[\frac{1}{2}|\nabla\phi|^2 + V(\bx)|\phi|^2 + \frac{\beta_n+\beta_s}{2}|\phi|^4\right]\dif\bx,$$
under constraint
$\mathcal S :=\left\{\phi \in L^2(\mathbb R^d) \big|
\|\phi\|^2 = 1, 
E_{\textrm{SMA}}(\phi) < \infty \right\}.$
\end{lem}

\begin{thm}[\textbf{Nullspace of $\cH_-$ and $\cH_+$ for ferromagnetic system}]\label{SMA}
Under conditions in Lemma \ref{lemma:ferr}, we derive the nullspace of $\mathcal H_-$ and  $\mathcal H_+$ as follows
\begin{equation}
	\label{equ:null_H_ferr}
\ns(\cH_-)=\spa \{\Phi_1,\Phi_2\},\qquad
\ns(\cH_+)=\spa \{\Phi_2\},
\end{equation}
where $\Phi_1 = \Phi_g={\bm a}\,\phi_g$, $\Phi_2 = {\bm b}\,\phi_g$ with ${\bm b} = \big((\frac{1-M^2}{2})^{1/2}, -M,-(\frac{1-M^2}{2})^{1/2}\big)^\top$.
\end{thm}
Moreover, $\mathcal H_-$ and $\mathcal H_+$ are positive semidefinite.
\begin{proof}

It is straightforward to check that
	\begin{equation*}
		\mathcal{H}_-  = \mathcal{A}-\mathcal{B}
		=
		\Big(L+(\beta_n+\beta_s)\phi_g^2 - \mu \Big) \mathbf{I}_3
		+ \beta_s\phi_g^2\Big(
             \bss\cdot \bbs({\bm a})
			+ \bss_y{\bm a}{\bm a}^{\top}\bss_y^{\sH}
		- \bss_y{\bm a}{\bm a}^{\top}\bss_y^{\top}
		-\mathbf{I}_3
	\Big),
	\end{equation*}
     where $\mu = \mu_1 = \mu_0 = \mu_{-1}$ in Lemma \ref{lemma:ferr}.
Moreover, we have
	\begin{equation*}
\mathcal{W}^{-1}
\Big(
\bss\cdot \bbs({\bm a})
+ \bss_y{\bm a}{\bm a}^{\top}\bss_y^{\tt H}
- \bss_y{\bm a}{\bm a}^{\top}\bss_y^{\top}
-\mathbf{I}_3
\Big)\mathcal{W}
=\begin{bmatrix}0&0&0\\0&0&0\\0&0&-2\end{bmatrix},
\end{equation*}
where $\mathcal{W}=[{\bm a},{\bm b},{\bm c}]$ is an orthogonal symmetric matrix with the third column ${\bm c} = \big(\frac{1-M}{2},-(\frac{1-M^2}{2})^{1/2},\frac{1+M}{2}\big)^{\top}$.
Then, we can diagonalize operator $\mathcal H_{-}$ using $\mathcal W$ as
$\mathcal W^{-1}\mathcal H_- \mathcal W = \mathcal K$, where $\mathcal K$ reads explicitly as follows
\beas
\mathcal K = \diag\Big(L + (\beta_n + \beta_s)\phi_g^2 - \mu, L + (\beta_n+\beta_s)\phi_g^2- \mu , L + (\beta_n - \beta_s)\phi_g^2- \mu\Big).
\eeas
According to Lemmas 1 \& 2 in \cite{Maday-JSC},
$\ns\big(L + (\beta_n + \beta_s)\phi_g^2 - \mu\big) = \spa\{\phi_g\}$, and $L + (\beta_n - \beta_s)\phi_g^2- \mu$ is symmetric positive definite.
Hence, we can derive nullspace of $\mathcal K$ as
    \begin{equation*}
      \ns(\mathcal K) = \spa\left\{
        \begin{bmatrix} 1 \\ 0 \\ 0 \end{bmatrix}\phi_g,\
        \begin{bmatrix} 0 \\  1 \\ 0 \end{bmatrix}\phi_g
      \right\}.
    \end{equation*}
Therefore, we obtain
    \begin{equation*}
		\ns(\cH_-)=\spa \{ \Phi_1,\Phi_2\}=\spa \{ {\bm a}\,\phi_g,{\bm b}\,\phi_g\} .
    \end{equation*}
Similarly, we have $$\ns(\cH_+)=\spa \{\Phi_2\}.$$ Proofs are similar and we choose to omit details for brevity.

\end{proof}

\begin{thm}[\textbf{Generalized nullspace of $\cH$ for ferromagnetic system}] \label{thm:ferromagnetic}
Under conditions in Lemma \ref{lemma:ferr},
we have the nullspace of $\mathcal H$ as
\begin{equation}\label{equ:ferr_H}
\ns(\cH) = \spa \left\{\begin{bmatrix}
\mathbf 0 \\
\Phi_1
\end{bmatrix},
\begin{bmatrix}
\mathbf 0 \\
\Phi_2
\end{bmatrix},
\begin{bmatrix}
\Phi_2 \\
\mathbf 0
\end{bmatrix}
\right\},
\end{equation}
and the generalized nullspace
\begin{equation}
\label{equ:ferr_H2}
\ns(\cH^3) = \ns(\cH^2) = \ns(\cH) \oplus \spa \left\{\begin{bmatrix}
\widehat\Phi_1\\
\mathbf 0
\end{bmatrix}
\right\},
\end{equation}
where $\widehat{\Phi}_1$, satisfying $\cH_+\widehat{\Phi}_1=\Phi_1$, is independent of $\Phi_2$.
\end{thm}
\begin{proof}
For any $(\bbf^{\top}, \bbg^{\top})^{\top}\in\ns(\cH)$, we have $\bbf\in\ns(\cH_+)$ and $\bbg\in\ns(\cH_-)$.
According to Theorem \ref{SMA}, we can prove Eqn. \eqref{equ:ferr_H}.
Next, we shall investigate the nullspace of $\cH^2$.
For any $(\bbf^{\top}, \bbg^{\top})^{\top}\in \ns(\cH^2)$, we have
\begin{equation}\label{H2Null}
	\begin{bmatrix}
		\mathcal O &\mathcal H_- \\
		\mathcal H_+&\mathcal O
	\end{bmatrix}
	\begin{bmatrix}
		\bbf  \\ \bbg
	\end{bmatrix} \in \ns(\cH).
\end{equation}
It is worthy to note that $\Phi_2$ does not belong to range of $\cH_+$, i.e., $\Phi_2\notin \mathcal R(\cH_+)$.
Otherwise, there exists a solution to $\cH_+\bbf=\Phi_2$, then we derive a contradiction
\begin{equation*}
0<\lag\Phi_2,\Phi_2\rag
=\lag\Phi_2,\cH_+\bbf\rag
=\lag\cH_+\Phi_2,\bbf\rag = 0.
\end{equation*}
Similarly, we can prove that $\Phi_2 \notin \mathcal R(\cH_{-})$.
Therefore, equation \eqref{H2Null} is equivalent to
	\begin{equation*}
\mathcal H_- \bbg = \mathbf{0},\quad
\mathcal H_+ \bbf = \Phi_1.
	\end{equation*}
Since $\cH_+$ is a self-adjoint compact operator,
using the spectral theory of operator, 
the fact that
$\Phi_1\not\in\ns(\cH_+)$ leads to
$\Phi_1\in \mathcal R(\cH_+)$.
That is, there exists a function $\widehat{\Phi}_1$ such that $\cH_+\widehat{\Phi}_1=\Phi_1$.
Then, we have
	\begin{equation*}
		\ns(\cH^2) = \ns(\cH) \oplus \spa \left\{\begin{bmatrix}
				\widehat\Phi_1\\
				\mathbf 0
			\end{bmatrix}
		\right\}.
	\end{equation*}

Finally, we study the nullspace of $\cH^3$.
Similarly, if $\ns(\cH^2) \subsetneq \ns(\cH^3)$, we have the following equivalence
	\begin{equation*}
		\begin{bmatrix}
			\mathcal O &\mathcal H_- \\
			\mathcal H_+&\mathcal O
		\end{bmatrix}
		\begin{bmatrix}
			\bbf  \\ \bbg
		\end{bmatrix}=
		\begin{bmatrix}
			\widehat\Phi_1\\  \mathbf{0}
		\end{bmatrix}\Longleftrightarrow
		\begin{cases}
\mathcal H_- \bbg = \widehat\Phi_1\\
\mathcal H_+ \bbf = \mathbf{0}
		\end{cases}.
	\end{equation*}
However, the above equation does not admit any solutions, and it can be proved using the following argument.
If there exists a solution $\bbg$ satisfying $\mathcal H_- \bbg = \widehat\Phi_1$, then we have the following contradiction
\begin{equation*}
0=\langle \mathcal H_-\Phi_1, \bbg \rangle
=\langle \Phi_1,\mathcal H_- \bbg \rangle
=\langle \Phi_1,\widehat\Phi_1 \rangle
=\langle \mathcal H_+\widehat\Phi_1,\widehat\Phi_1 \rangle>0.
\end{equation*}
That is to say, $\ns(\cH^3) = \ns(\cH^2)$, and proof is completed.
\end{proof}

\vspace{0.25cm}

\noindent {\bf Case II: Anti-ferromagnetic system ($\beta_s>0$)}.

\vspace{0.1cm}

\begin{lem}[{\cite[Theorem4.3]{Bao2017Mathematical}}]\label{lem2Comp}
If $M\in (-1,1)$, $\beta_s>0$ and $\beta_n\geq 0$,
there exists a ground state solution $\Phi_g = (\phi_1^g,0,\phi^g_{-1})^\top$
and $(\phi^g_1,\phi^g_{-1})$ is a minimizer of the energy functional
$$E_{0}\big(\phi_1,\phi_{-1}\big)
= \int_{\mathbb{R}^d}\left[\sum_{j=\pm 1}\bigg(\frac{1}{2}|\nabla\phi_j|^2 + V(\bx)|\phi_j|^2\bigg) + \frac{\beta_n+\beta_s}{2}(|\phi_1|^4+|\phi_{-1}|^4) + (\beta_n-\beta_s)|\phi_1|^2|\phi_{-1}|^2\right]\dif\bx$$
under constraint
$\mathcal S :=\left\{(\phi_1,\ \phi_{-1}) \big|~
\|\phi_1\|^2=\frac{1+M}{2},\ \|\phi_{-1}\|^2=\frac{1-M}{2},\
E_{0}(\phi_1,\phi_{-1}) < \infty \right\}.$
\end{lem}

\begin{thm}[\textbf{Nullspace of $\cH_-$ for anti-ferromagnetic system}]\label{TCC}
   Under conditions in Lemma \ref{lem2Comp}, we have the following property for nullspace of $\mathcal H_-$
$$\spa \{\Phi_1,\Phi_2\}\subset \ns(\cH_-),$$
where $\Phi_1 = \Phi_g$, $\Phi_2 = (b_{1}\phi_1^g,0,b_{-1}\phi_{-1}^g)^\top$, and
 $(b_{1},b_{-1}) = \big(-(\frac{1-M}{1+M})^{1/2}, (\frac{1+M}{1-M})^{1/2}\big)$.
\end{thm}

\begin{proof}
Noticing the fact that $\phi_0 = 0$, we can reduce Eqn. \eqref{NLEP} as follows
\begin{equation}\label{equ:anti_ferr_gs}
\left\{\begin{array}{ll}
   [ L -\ \mu_1 \ + \beta_n\rho + \beta_s(|\phi^g_1|^{2}-|\phi^g_{-1}|^{2})]~ \phi^g_1& = 0,\\[0.5em]
   [L -\mu_{-1} +  \beta_n\rho - \beta_s(|\phi^g_1|^{2}-|\phi^g_{-1}|^{2})]~\phi^g_{-1}&= 0,
\end{array}
\right.
\end{equation}
and the operator $\mathcal H_{-}$ is simplified below
\beas
&&\mathcal H_- = \diag\Big( L -\mu_1+ \beta_n\rho + \beta_s\big(|\phi^g_1|^{2}-|\phi^g_{-1}|^{2}\big),\ \  L - \mu_{0} + \beta_n\rho + \beta_s\big(\phi^g_{1}-\phi^g_{-1}\big)^2,\\
&&\qquad\qquad\quad\ L -\mu_{-1} +  \beta_n\rho - \beta_s\big(|\phi^g_1|^{2}-|\phi^g_{-1}|^{2}\big) \Big).
\eeas

Then, according to Eqn. \eqref{equ:anti_ferr_gs},
we obtain that $\Phi_1$ and $\Phi_2$ lie in the nullspace of $\cH_-$, i.e., $\cH_-\Phi_1 = 0$ and $\cH_-\Phi_2 = 0$.
It is easy to check that
$\lag\Phi_1,\Phi_2\rag = 0$ and $\|\Phi_2\| = 1$ with $b_{1}= -(\frac{1-M}{1+M})^{1/2},~b_{-1} = (\frac{1+M}{1-M})^{1/2}$.
\end{proof}

\begin{remark}\label{rem:anti-ferr}
From our extensive numerical results not fully shown here,
we conjecture that the following property holds for anti-ferromagnetic systems, that is,
\begin{center}
   $\cH_+$ is invertible ~ ~ ~ and ~ ~ ~ $\ns(\cH_-) = \spa \{\Phi_1,\Phi_2\}$.
\end{center}
Moreover, $\mathcal H_+$ is positive definite. and $\mathcal H_-$ is positive semidefinite.

If the above property is true,
we can obtain the generalized nullspace of $\mathcal H$
as follows
\bea
\ns(\cH^3) = \ns(\cH^2) = \ns(\cH) \oplus \spa \left\{\begin{bmatrix}
\widehat\Phi_1\\
\mathbf 0
\end{bmatrix},
\begin{bmatrix}
\widehat\Phi_2\\
\mathbf 0 \\
\end{bmatrix}\right\},
\eea
where $\widehat{\Phi}_j = \cH_+^{-1}\Phi_j, ~j=1,2$ and
\begin{equation*}
\ns(\cH) = \spa \left\{\begin{bmatrix}
\mathbf 0 \\
\Phi_1
\end{bmatrix},
\begin{bmatrix}
\mathbf 0 \\
\Phi_2
\end{bmatrix}\right\}.
\end{equation*}

\end{remark}


The convergence of the non-zero eigenvalues depends heavily on the approximation accuracy of generalized nullspace that is associated with $\cH$ \cite{LiWangZhang}.
Based on Theorem \ref{thm:ferromagnetic} and Remark \ref{rem:anti-ferr},  we obtain the generalized nullspace of $\cH$, thus essentially improving the convergence and efficiency.

\section{Numerical method}
\label{Numer_alg}

In this section, we will propose an efficient and spectral accurate numerical method to solve the BdG equations \eqref{HpHm-eq}.
Due to the presence of external trapping potential $V(\bx)$, the ground states $\Phi_g(\bx)$
and the eigenfunctions $(\bu,\bv)$/$(\bbf,\bbg)$
are all smooth and fast decaying. Therefore, it is reasonable
to truncate the whole space $\mathbb{R}^d$
into a bounded domain $\Omega \subset \mathbb{R}^d$ that is large enough such that truncation error is negligible.
Since all related functions are smooth,  the Fourier pseudospectral (FS) discretization stands out as the optimal candidate for spatial discretization \cite{Bao2017Mathematical,Bao2007mass,Tang2022Spectrally},
due to its simplicity, spectral accuracy, and great efficiency that is guaranteed by discrete Fast Fourier transform (FFT).
The computation domain is usually chosen as a rectangle, denoted as ${\mathcal D}_L:=[-L,L]^{d}$.  Spatial discretization details will be presented in the next subsection.

\subsection{Spatial discretization by Fourier spectral method}\label{FS_d}
Provided that the stationary states $\Phi_g$ and all chemical potentials are precomputed with a very fine mesh by the PCG method \cite{ATZ-CiCP-PCG}, such that the numerical accuracy approaches machine precision.
For simplicity, we choose to illustrate the spatial discretization for the 1D case and extensions to higher dimensions are omitted here.
We choose the computational domain ${\mathcal D}_L:=[-L,L]$ and discretize it uniformly with mesh size $h_x=\fl{2L}{N}$ where $N$ is a positive even integer. Define the grid point set as $\mathcal{T}_{x}=\{(-N/2,\cdots,N/2)h_x\}$
and introduce the plane wave basis and discrete space as
\bes
W_{k}(x)=e^{\im \mu_k (x+L)}\quad (-N/2 \leq k \leq N/2-1), \mbox{ and } X_N = \spa\{W_{k}(x)\}_{k=-N/2}^{N/2-1},
\ees
with $\mu_k = \pi  k/L $.
Define $F(x_n)$ ($F = \phi^g_j, u_j, v_j, V$ etc.) as function value at grid point $x_n = -L+nh_x \in \mathcal{T}_{\bx}$ and  $\mathbf F = \big(F(x_1),F(x_2),\cdots,F(x_N)\big)$ as the corresponding discrete vector.
The Fourier pseudo-spectral approximations of $F$, denoted by ${F_N}$, and its Laplacian $\nabla^2 F$ read as follows
\bea\label{laplace_approx}
F(x) \approx (\mathcal{I}_N{F})(x):=\sum_{k=-N/2}^{N/2-1}\widetilde{{F}}_{k}\; W_{k}(x),
\quad (\nabla^2 F)(x) \approx (\nabla^2 (\mathcal{I}_N{F}))(x)\! =\sum_{k=-N/2}^{N/2-1}-\mu_k^2~\widetilde{{F}}_{k}\; W_{k}(x),
\eea
where $\mathcal I_N$ denotes a mapping from $C(\Omega)$ to $X_N$, and $\widetilde{{F}}_{k}$, the discrete Fourier transform of  vector $\mathbf F$, is computed as
\begin{equation}
\label{eq:FT_3}
 \widetilde{F}_{k}= \frac{1}{N} \sum_{n=0}^{N-1} F(x_n)\overline{W}_{k}(x_n)
 = \frac{1}{N} \sum_{n=0}^{N-1} F(x_n) ~e^{\frac{-i 2\pi n k}{N}}, \quad \quad    -N/2 \leq k \leq N/2-1,
  \end{equation} and is accelerated by discrete Fast Fourier transform (FFT) within  $O(N \ln(N))$ float operations. The numerical approximation of Laplacian operator $\nabla^2$, denoted as $[\![\nabla^2]\!]$, corresponds to a dense matrix \cite{SpectralBkShen,Tang2022Spectrally,Zhang2021}.
 We map the function's pointwise multiplication $F_{n} G_{n}$ as a matrix-vector product $[\![F]\!] {\mathbf G}$, that is,
  \bea
\left( [\![F]\!] {\mathbf G} \right)_{n} := F_{n}G_{n} \quad \Longrightarrow \quad [\![F]\!] = \diag(\mathbf F).
 \eea
Therefore, operators $\cA$ and $\cB$ are mapped into matrices $\mathbf A$ and $\mathbf B$ in the following way
\bea \label{def_bfA}
&&\mathbf A := [\![L\mathbf I_3-\Lambda]\!] + \beta_n\Big([\![\rho\mathbf{I}_3]\!] + [\![\Phi_g\Phi_g^\top]\!]\Big) + \beta_s\bigg([\![\bss\cdot\bbs(\Phi_g)]\!] + \sum_{j=x,y,z}[\![\bss_j\Phi_g\Phi_g^\top\bss_j]\!]\bigg), \\
\label{def_bfB}
&&\mathbf B:= \beta_n[\![\Phi_g\Phi_g^\top]\!] + \beta_s\sum_{j=x,y,z}[\![\bss_j\Phi_g\Phi_g^\top\bss_j^\top]\!].
\eea
Thanks to the Fourier spectral discretization, both matrices $\mathbf A$ and $\mathbf B$ are symmetric, thus keeping their continuous operators' Hermitian property.
Setting $\bm{u}_N = (\bm{u}_{1,N};\bm{u}_{0,N};\bm{u}_{-1,N})$ and $\bm{v}_N = (\bm{v}_{1,N};\bm{v}_{0,N};\bm{v}_{-1,N})$,
the BdG equations \eqref{BdG_new} are then discretized into a linear eigenvalue problem
\begin{gather}
\label{BdG-d}
\begin{bmatrix}
\mathbf{A}&\mathbf{B} \\
-\mathbf{B}&-\mathbf{A}
\end{bmatrix}
\begin{bmatrix}
\bm{u}_N \\
\bm{v}_N
\end{bmatrix}
=\omega_N\begin{bmatrix}
\bm{u}_N\\
\bm{v}_N
\end{bmatrix}.
\end{gather}
With a similar change of variables, i.e.,
\be\label{c-o-v-h}
\bm{f}_N = \frac{1}{2}(\bm{u}_N+\bm{v}_N),~\bm{g}_N = \frac{1}{2}(\bm{u}_N-\bm{v}_N),
\ee
the above equation is rewritten into a linear
response eigenvalue problem
\begin{gather}
\label{lrep-d}
\begin{bmatrix}
{\mathbf O}&\mathbf{H}_{-} \\
  \mathbf{H}_{+}&{\mathbf O}
\end{bmatrix}
\begin{bmatrix}
\bm{f}_N \\
\bm{g}_N
\end{bmatrix}
=\omega_N\begin{bmatrix}
\bm{f}_N\\
\bm{g}_N
\end{bmatrix},
\end{gather}
where $\mathbf{H}_{+} = \mathbf{A}+\mathbf{B}$, $\mathbf{H}_{-} = \mathbf{A}-\mathbf{B}$.

The above discrete dense eigensystem \eqref{lrep-d} is solved using the recently developed BiOrthogonal Structure Preserving algorithm (BOSP for short) \cite{LiWangZhang}, and details are illustrated in the coming subsection.

\begin{remark}[Matching eigenfunction constrain \eqref{constrainfg}]
In practice, we {\sl do not}  need to bother about the constrain \eqref{constrainfg} when solving the eigenvalue problem \eqref{lrep-d}.
Actually, once the eigenvectors are computed, the constrain \eqref{constrainfg} will be satisfied easily with a scalar scaling of $\bbf_N$ and $\bbg_N$.
\end{remark}

\subsection{Bi-orthogonal structure-preserving Fourier spectral eigensolver}\label{subsec:LREP-biorth}

The ARPACK package has been successfully adapted to BdG equations of dipolar BEC \cite{ARPACK,Tang2022Spectrally}.
However, a similar adaptation attempt is not as effective, and sometimes it does not even converge within a reasonable time,
because the generalized nullspace is much larger and more complicated, not to even mention the large-scale dense eigensystem for three-dimensional problems.
The eigenspaces associated with eigenvalues of different magnitudes are biorthogonal, and such biorthogonality property shall be taken into account in the eigensolver design \cite{Bai2012Minimization,Bai2013Minimization,LiWangZhang}.
The structure of generalized nullspace and biorthogonality, well described in Section \ref{sec:BdG_Prop} on the continuous level,
are assessed numerically by the linear algebra package (LAPACK) and a modified Gram-Schmidt biorthogonal algorithm.
Based on the recently developed linear response eigensolver BOSP \cite{LiWangZhang},
by combing the Fourier spectral method for spatial approximation and making use of the specific generalized nullspace structure (Section \ref{sec:rootspace}),
we propose an efficient iterative subspace eigensolver and provide a friendly interface for matrix-vector product evaluation.
Since there is no explicit matrix storage and the matrix-vector product is implemented via FFT within almost optimal $O({\tt DOF}\log({\tt DOF}))$ operations ({\tt DOF} is the degree of freedom/total number of grid points, i.e., ${\tt DOF} = N^d$),
the heavy memory burden is much more alleviated and the computational efficiency is guaranteed to a large extent.
Therefore, it provides a feasible solution to large-scale problems, especially in such a densely populated system in our case.
Thanks to the Fourier spectral method, our solver can achieve spectral accuracy for both eigenvalue and eigenvectors as long as the BOSP's accuracy tolerance is chosen sufficiently small.

As is known, the matrix-vector product evaluation is the most time-consuming part. While, in this article, the matrix-vector product
$$\mathbf{H}_+ \bm{f}_N,\quad \mathbf{H}_-\bm{g}_N,\quad \mbox{for} \quad \bm{f}_N, \bm{g}_N \in  {\mathbb R}^{3N},$$
can be realized by two pairs of FFT/iFFT plus some function multiplication in physical space.
The overall computational costs to compute the first {\tt nev} eigenpairs amounts to
\beas
\mathcal{O}({\tt nev})\times \big(\mathcal{O}({\tt DOF} \log({\tt DOF})) + \mathcal{O}({\tt DOF})\big) + \mathcal{O}({\tt nev}^3)
=  \mathcal{O}({\tt nev}~N \log({\tt DOF}))  + \mathcal{O}({\tt nev}^3).
\eeas
When the degree of freedom is much larger than the number of eigenvalues, i.e., $ {\tt DOF} \gg {\tt nev}$, the computational costs are approximately $\mathcal{O}({\tt nev}\times{\tt DOF} \log({\tt DOF}))$ and will be verified numerically in the next section.

\subsection{Stability analysis}
To demonstrate the practicality of our solver,  we will present a rigorous stability analysis for the eigenfunctions in this subsection.
Since the density and eigenfunction both decay fast enough in space,
it is reasonable to assume that $\chi(\bx)$ ($\chi = \phi^g_j, u_j, v_j$, etc.) is numerically compactly supported
in a bounded domain $\Omega\subsetneq\mathbb{R}^d$, that is,  $\mathrm{supp}\{\chi\} \subsetneq \Omega$. We introduce $2L$-periodic Sobolev space $H^m_p(\Omega)
\subset H^m(\Omega) \ \ (m\geq1)$ with $\Omega = \mathcal{D}_L$, and the semi-norm, norm and $\infty$-norm as follows, respectively,
$$|\chi|_{m} := \left(\sum\nolimits_{|\bm{\alpha}|=m} \|\partial^{\bm{\alpha}} \chi \|^2\right)^{1/2},
\quad  \|\chi\|_{m} := \left(\sum\nolimits_{ |\bm{\alpha}|\leq m} \|\partial^{\bm{\alpha}} \chi \|^2\right)^{1/2},\quad
\|\chi\|_{\infty}=
\sup_{\bx\in\Omega}|w(\bx)|,$$
with index $\bm\alpha=(\alpha_1,\ldots, \alpha_d)\in\mathbb{Z}^d$, $|\bm\alpha|=\sum_{j=1}^d\alpha_j$, $\partial^{\bm\alpha}=\partial_{x_1}^{\alpha_1}\cdots\partial_{x_d}^{\alpha_d}$ and $\|\cdot\|$ being the $L^2$ norm.
For vector-valued function $\bm \chi = (\chi_{1},\chi_0,\chi_{-1})^\top$, we define $|\bm \chi|_{m} := \sqrt{\sum_{j=-1}^1|\chi_j|_{m}^2}$,~ $\|\bm \chi\|_{m} := \sqrt{\sum_{j=-1}^1\|\chi_j\|_{m}^2}$,~$\|\bm \chi\| :=\sqrt{\sum_{j=-1}^1\|\chi
_j\|^2}$,~ and $\|\bm \chi\|_{\infty} := \max\{\|\chi_j\|_{\infty}\}$. We use $A\lesssim B$ to denote $A \leq cB$ where the
constant $c>0$ is independent of the grid number $N$. Then we introduce the stability analysis as follows:

\begin{thm}[\textbf{Stability Analysis}]\label{thm:stability}
 For any $\bm{\chi}(\bx)=(\chi_1(\bx),\chi_0(\bx),\chi_{-1}(\bx))^\top \in (H_p^{m}(\Omega))^3$  with $m> d/2 + 2$, and its Fourier pseudo-spectral approximations \\ $(\cI{\bm{\chi}})(\bx)=\big((\cI{\chi_1})(\bx),(\cI{\chi_0})(\bx),(\cI{\chi_{-1}})(\bx)\big)^\top$, we have the following error estimate
\bea
\|\mathcal{Q} \bm{\chi}- \mathcal{Q}(\cI{\bm{\chi}})\|_{\infty} & \lesssim & N^{-(m - \frac{d}{2} - 2)},
\eea
where $\mathcal{Q} = \mathcal{A},\mathcal{B}$, and $\bm{\chi}=\bm{u},\bm{v}$.
\end{thm}
To prove the above theorem, the following preparations are required.
\begin{lem}[\textbf{Fourier spectral approximation} \cite{LiuZhang_opt,SpectralBkShen}] \label{FouApprox}
 For $\chi(\bx) \in H_p^{m}(\Omega)$  with $m> d/2$, and its Fourier pseudo-spectral approximation $\cI \chi$,
 we have the following error estimates
\bea
\| \partial^{\bm{\alpha}}( \chi- \cI \chi )\|_{\infty} & \lesssim & N^{-(m -\frac{d}{2}- |\bm{\alpha}|)} |\chi|_{m} ,\quad\quad ~0\leq |\bm\alpha|\leq m, \label{approx_lem}\\
\| \partial^{\bm{\alpha}}( \chi- \cI\chi )\| & \lesssim & N^{-(m - |\bm{\alpha}|)} |\chi|_{m} ,\quad\quad ~0\leq |\bm\alpha|\leq m. \label{approx_lem_2}
\eea
\end{lem}

Since operators $\mathcal A$ and $\mathcal B$ contain only two types of operators, namely, the Laplacian operator $\nabla^2$
and function multiplication operation that involves  $V(\bx)$ \& $\phi_i^g\phi_j^g$, the following corollary is  a prerequisite for proving Theorem \ref{thm:stability}.
\begin{corollary}\label{coro_approx}
 For $\chi(\bx) \in H_p^{m}(\Omega)$  with $m> d/2$, and its Fourier pseudo-spectral approximation $\cI \chi$, we have
\bea
\|\nabla^2 \chi- \nabla^2(\cI \chi)\|_{\infty} & \lesssim & N^{-(m -\frac{d}{2}- 2)} |\chi|_{m},\\ \label{error_1}
\|V \chi- V(\cI\chi)\|_{\infty}  & \lesssim & N^{-(m -\frac{d}{2})} |\chi|_{m},\\ \label{error_2}
\|{\phi_i^g}{\phi_j^g} \chi- {\phi_i^g}{\phi_j^g}(\cI\chi)\|_{\infty}  & \lesssim & N^{-(m -\frac{d}{2})} |\chi|_{m} \ \ (i,j=1,0,-1).\label{error_3}
\eea
\end{corollary}
Sine it is easy to prove the above corollary using Lemma \ref{FouApprox} and we shall omit details for brevity.

\vspace{0.2cm}

\noindent {\bf The proof of Theorem \ref{thm:stability}}:
\vspace{-0.2cm}
\begin{proof}
For simplicity, we only prove the $\mathcal Q=\mathcal A$ case and extensions to other operators are similar.
From \eqref{approx_lem}-\eqref{error_2}, we obtain
$$\|(L\mathbf{I}_3-\Lambda)\bm \chi-(L\mathbf{I}_3-\Lambda)(\cI{\bm \chi})\|_{\infty}\lesssim N^{-(m-\frac{d}{2}-2)}|\bm \chi|_{m}.$$

For the second and third terms of $\mathcal A$ as shown in Eqn. \eqref{def_A} (the same as  $\mathbf A$ in \eqref{def_bfA}), using \eqref{approx_lem}, \eqref{error_3} and triangle inequality,
the following error estimate holds:
\beas
\left\|\big(\mathcal{A}-(L\mathbf{I}_3-\Lambda)\right)\bm \chi - \left(\mathcal{A}-(L\mathbf{I}_3-\Lambda)\big)(\cI{\bm \chi})\right\|_{\infty} \lesssim & N^{-(m-\frac{d}{2})}|\bm \chi|_{m}.
\eeas
Therefore, we finish the proof of Theorem \ref{thm:stability}.
\end{proof}

\subsection{Convergence analysis}
In this subsection, we will propose the convergence analysis of Fourier pseudo-spectral approximations. To begin with, and adopt the Sobolev space $\mathbb{V} = (H^1_p(\Omega))^3\times (H^1_p(\Omega))^3$ with the following norm
$$\|\bm \Phi\|_{1} = \sqrt{\|\bm f\|_{1}^2+\|\bm g\|_{1}^2}~\mbox{ and }~ \|\bm \Phi\| = \sqrt{\|\bm f\|^2+\|\bm g\|^2}, \quad ~\forall \,\bm\Phi:= (\bm f;\bm g) \in \mathbb V.$$

The BdG equation \eqref{HpHm-eq} is equivalent to the following eigenvalue problem:
\bea \label{HpHm-eq-1}
\widetilde\cH\begin{bmatrix}
	\bbf  \\ \bbg
	\end{bmatrix} :=
\begin{bmatrix}
\mathcal H_+&\mathcal O\\
	\mathcal O&\mathcal H_-
\end{bmatrix}
\begin{bmatrix}
	\bbf  \\ \bbg
	\end{bmatrix} =
	\omega \begin{bmatrix}
0&1\\
	1&0
\end{bmatrix}
\begin{bmatrix}
	 \bbf  \\ \bbg
\end{bmatrix}=:
	\omega \mathcal J
\begin{bmatrix}
	 \bbf  \\ \bbg
\end{bmatrix}.
\eea
The Galerkin weak form of Eqn.~\eqref{HpHm-eq-1} reads as: to find $0\neq \omega\in \mathbb{R}$ and $\bm \Phi=(\bm f;\bm g) \in \mathbb{V}$ such that
\bea\label{BdG-wf}
a(\bm \Phi,\bm\Psi) = \omega\, b(\bm \Phi,\bm\Psi), \ \ \ \forall\, \bm\Psi= (\bm\xi;\bm\eta)\in \mathbb{V},
\eea
subject to constraint $b(\bm \Phi,\bm \Phi) = 1/2$ with
\beas
a(\bm \Phi,\bm\Psi) := \lag \cH_+ \bm f,\bm\xi\rag  + \lag \cH_-\bm g,\bm\eta\rag, \quad b(\bm \Phi,\bm\Psi) := \lag\bm g,\bm\xi\rag + \lag\bm f,\bm\eta\rag,
\eeas
where $\lag \bm f, \bm g\rag:=\sum_{j=-1}^1\int_{\mathbb{R}^d} {f_j}(\bx)\overline{g_j(\bx)}\, \dif \bx,~ \forall \bm f,\bm g \in (L^2(\Omega))^3$.

\begin{remark}
From Theorem \ref{thm:ferromagnetic} or Remark \ref{rem:anti-ferr}, we can prove that there exists $c_0>0$ such that
\begin{equation}\label{coer_BdG}
a(\bm \Phi,\bm\Phi) \geq c_0\|\bm \Phi\|_{1},\quad \forall\, \bm\Phi \in \mathbb{U},
\end{equation}
where the subspace $\mathbb{U}$ is defined as
$$ \mathbb{U} = \ns(\cH_+)^\bot\times \ns(\cH_-)^\bot \subset \mathbb{V},$$
with $\ns(\cH_\pm)^\bot := \{\bm f \in (H_p^1(\Omega))^3\, |\, \lag\bm f,\bm\xi\rag = 0, ~\forall\, \bm\xi \in \ns(\cH_\pm)\}$.
It is important to point out that solving the BdG equation \eqref{BdG-wf} in $\mathbb{U}$ is equivalent to solving non-zero $\omega$ and its associated eigenfunction $\bm{\Phi}$ in $\mathbb{V}$.
\end{remark}
Define the approximation finite dimensional spaces $X_N$ of $H^1_p(\Omega)$ and $\mathbb{V}_N$ of $\mathbb{V}$ as follows:
$${X}_N:= \spa\left\{W_k(\bx), k = -N/2,\ldots, N/2-1\right\},\quad \mathbb{V}_N := ({X}_N)^3\times ({X}_N)^3,$$
and consider its approximation problem: to find $0\neq \omega_N \in \mathbb{R}$ and $\bm \Phi_N=(\bm f_N;\bm g_N)\in \mathbb{V}_N$ such that
\bea\label{BdG-wf-a}
a_N(\bm \Phi_N,\bm\Psi_N) = \omega_N\,b_N(\bm \Phi_N,\bm\Psi_N), \ \ \ \forall\, \bm\Psi_N=(\bm \xi_N;\bm \eta_N)\in \mathbb{V}_N,
\eea
subject to normalization constraint $b_N(\bm \Phi_N,\bm\Phi_N)={1}/{2}$ with
\beas
a_N(\bm \Phi_N,\bm\Psi_N) = \lag\cH_+\bm f_N,{\bm\xi_N}\rag_N+ \lag\cH_-\bm g_N,{\bm\eta_N}\rag_N, \quad
b_N(\bm \Phi_N,\bm\Psi_N) = \lag\bm g_N,{\bm\xi_N}\rag_N + \lag\bm f_N,{\bm\eta_N}\rag_N,
\eeas
 and  $\lag\bm f,\bm \xi\rag_N:= \sum_{j=-1}^1\lag f^N_j,\xi^N_j\rag_N := \sum_{j=-1}^1 \big(\frac{2L}{N}\sum_{n=0}^N{f_j^N}(\bx_n)\overline{g_j^N(\bx_n)}\big)$.

 \

\noindent Similar as Lemma A.\ref{lem:equivalent}, we have the following equivalence result and choose to omit proofs for brevity.
\begin{lem}
The discrete problem \eqref{lrep-d} and discrete variational problem \eqref{BdG-wf-a} are equivalent.
\end{lem}

To illustrate the convergence behavior, we introduce the following notation
\begin{equation}\label{infSupCondVecFun}
\delta_N(\bm\Phi) = \inf_{\bm\Psi\in \mathbb{V}_N}\bigg\{\|\bm\Phi-\bm\Psi\|_1 + \sup_{\bm \Theta\in \mathbb{V}_N}\frac{|a(\bm\Psi,\bm \Theta)-a_N(\bm\Psi,\bm \Theta)|}{\|\bm \Theta\|_1} \bigg\}.
\end{equation}
Similar as Lemma A.\ref{lem_convergence} for scalar function,
by the coerciveness of bilinear operator $a(\cdot,\cdot)$ (Eqn. \eqref{coer_BdG}) and the conformal subspace approximation, we obtain
\begin{equation}
\delta_N(\bm\Phi)\lesssim N^{-(m-\sigma)},\qquad \mbox{with}  \quad \sigma = \max\{1,d/2\}.
\end{equation}
According to the general theory given by Lemma A.\ref{lem:ErrEsti}, we derive the following error estimates.

\begin{lem}\label{lem:ErrEsti1}
For any eigenpair approximation $\{\omega_N;\bm\Phi_N\}$ of Eqn. \eqref{BdG-wf-a}, there is
an eigenpair $\{\omega;\bm \Phi\}$ of Eqn. \eqref{BdG-wf} corresponding to $\omega$ such that
\beas
\|\bm \Phi-\bm\Phi_N\|_{1}\lesssim \delta_N(\bm\Phi),\quad
\|\bm \Phi-\bm\Phi_N\| \lesssim \zeta_N\|\bm \Phi-\bm\Phi_N\|_{1},
\quad |\omega-\omega_N| \lesssim \|\bm \Phi-\bm\Phi_N\|_{1}^2,\label{convee}
\eeas
where the constant $\zeta_N$ approaches $0$ as $N\rightarrow \infty$.
\end{lem}

Finally, via a change of variables, i.e., Eqn \eqref{c-o-v} and \eqref{c-o-v-h}, we obtain  error estimates for $(\bm u,\bm v)$.
\begin{thm}[\textbf{Error Estimates}]\label{thm:error-estimate}
If $\bm u,\bm v \in (H_p^{m}(\Omega))^3$  with $m\geq1$ and $\mathrm{supp}\{\bu\}, \mathrm{supp}\{\bv\} \subsetneq \Omega$, then for any eigenpair approximation $\{\omega_N;\bm{u}_N,\bm{v}_N\}$ of \eqref{BdG-d}, there is
an eigenpair $\{\omega;\bm u,\bm v\}$ of Eqn.~\eqref{BdG_new} satisfying the
following error estimates
\beas
\|\bm{u}-\bm{u}_N\|+\|\bm{v}-\bm{v}_N\| &\lesssim& N^{-(m-\sigma)}, \\
|\omega-\omega_N| &\lesssim& N^{-2(m-\sigma)},
\eeas
where $\sigma = \max\{1,d/2\}$.
\end{thm}
%
%

\section{Numerical results}
\label{NumResult}
In this section, we first carry out comprehensive numerical studies to illustrate the accuracy and efficiency of our solver. Then we apply it to investigate the Bogoliubov excitations around the ground states of spin-1 BEC.
The ground states $\Phi_g$ and the chemical potentials are precomputed with accuracy close to machine precision
 via PCG method \cite{ATZ-CiCP-PCG} in a large enough domain with a small enough mesh size.
Unless stated otherwise, we choose the potential $V(\bx)$ as the harmonic trapping potential \eqref{Vpoten}
and set the computational domain as squares and $L= 16$ for 1D \& 2D and $L= 8$ for 3D respectively, i.e., $\mathcal{D}=[-16, 16]$ in 1D, $\mathcal{D}=[-16, 16]^2$ in 2D and $\mathcal{D}=[-8, 8]^3$ in 3D respectively.
The domain is discretized uniformly in each spatial direction with the same mesh size $h=2L/N$ for simplicity.
We choose to study the ferromagnetic (\textbf{FM} for short) case with $\beta_n=885.4, ~\beta_s=-4.1$ and antiferromagnetic (\textbf{Anti-FM} for short) case with $\beta_n=240.8, ~\beta_s=7.5$ throughout the whole section.

\subsection{Performance investigation}
\label{Accuracy tests}
Firstly we verify the spectral accuracy for different space dimensions.
Analytical eigenvalues and eigenvectors were given for harmonic trapping potential by equation \eqref{eq:eig-pair} in {Theorem \ref{lemAnalytical}}.
For an eigenvalue $\og$ with multiplicity $k$, we denote its associated analytical eigenspaces as $\mathcal M_{\bm{u}}:={\rm span} \{\bm{u}_1, \cdots, \bm{u}_k\}$, $\mathcal M_{\bm{v}}:={\rm span} \{\bm{v}_1, \cdots, \bm{v}_k\}$, and denote $\{\omega_{\alpha,N};\bm{u}_{\alpha,N}, \bm{v}_{\alpha,N}\}$ as the numerical approximations of eigenvalue $\omega_{\alpha}$ and eigenfunctions $(\bm{u}_\alpha, \bm{v}_\alpha)$ obtained with mesh size $h= 2L/N$. To demonstrate the convergence,  we adopt the following error functions
\[
e_{\og_\alpha}^h:= \fl{|\omega_{\alpha,N}-\omega_\alpha|}{|\omega_\alpha|},
\qquad
e_{\bm{uv}}^{h,\alpha}:=\fl{\| \bm{u}_{\alpha,N}-\mathcal{P}_{\bm{u}}  \bm{u}_{\alpha,N}\|_2}{\| \bm{u}_{\alpha,N} \|_2}
+ \fl{\| \bm{v}_{\alpha,N}-\mathcal{P}_{\bm{v}} \bm{v}_{\alpha,N}\|_2}{\|\bm{v}_{\alpha}^h\|_2},
\] where $\alpha=x$ in 1D, $\alpha=x, y$ in 2D and $\alpha=x, y,z$ in 3D,
$\|\cdot\|_2$ is the discrete $l^2$ norm and  $\mathcal{P}_\nu$ ($\nu=\bm{u},\bm{v}$)
is the $l^2$-orthogonal projection operator into space $\mathcal M_\nu$.
 In both examples, we compute the first $40$ eigenvalues and their eigenfunctions.

\begin{exam}[\textbf{Accuracy}]
\label{Exam_Accuracy_Test}
We study the accuracy convergence for both ferromagnetic and antiferromagnetic cases in 1D/2D/3D.
To this end, we consider the following four cases
\begin{itemize}
\item[] \hspace{-0.5cm}{\bf Case I.}  1D case: $\gamma_x = 1$.
\item[] \hspace{-0.5cm}{\bf Case II.}  Isotropic 2D case:  $\gamma_x = \gamma_y=1 $.
\item[]\hspace{-0.65cm} {\bf Case III.} Anisotropic 2D case: $\gamma_x =\gamma_y/2=1$.
\item[] \hspace{-0.5cm}{\bf Case IV.} Isotropic 3D case:  $\gamma_x=\gamma_y=\gamma_z=1 $.
\end{itemize}

\end{exam}

\noindent

For {\bf Case I}, there exists one analytical eigenvalue $\og_x=1$ with multiplicity $k = 1$.
For {\bf Case II}, there exist two analytical eigenvalues $\og_x=\og_y=1$ with multiplicity $k = 2$.
For {\bf Case III}, $\og_x=1$ and $\og_y=2$ are eigenvalues with the same multiplicity $k = 1$.
Similarly, for {\bf Case IV}, there exist three analytical eigenvalues $\og_x=\og_y=\og_z=1$ with $k=3$.
Table \ref{Err_1D}-\ref{Err_3D} illustrate the numerical errors of eigenvalues and eigenvectors computed with different mesh sizes $h$
in {\bf Case I--IV}.

\begin{table}[H]
\def\temptablewidth{1\textwidth}
\tabcolsep 0pt	\caption{Errors of the eigenvalue/eigenvector for {\bf Case I} in \textbf{1D}
for \textbf{FM} (upper) and \textbf{Anti-FM} cases (lower) in Example \ref{Exam_Accuracy_Test}.}
	\label{Err_1D}\small
	\begin{center}\vspace{-1.5em}
{\rule{\temptablewidth}{1pt}}
\begin{tabularx}{\temptablewidth}{@{\extracolsep{\fill}}c|cccccccc}
   & $ $ &$h_0=1$& $h_0/2$ & $h_0/4$ & $h_0/8$ & $h_0/16$\\[0.2em]
\hline
	\rule{0pt}{12pt}
\multirow{2}{*}{ \textbf{~FM~}}   &  $ e_{\og_x}^h $              &4.392E-03 & 9.570E-05 & 2.422E-08 & 1.383E-09 & 1.648E-12\\[0.1em]
	 & $ e_{\bm{uv}}^{h,x} $    &1.069E-01 & 1.234E-03 & 2.911E-06 & 4.481E-08 & 1.165E-12
 \\ [0.2em]
\hline
	\rule{0pt}{12pt}
\multirow{2}{*}{ \textbf{~Anit-FM~~} }   & $ e_{\og_x}^h $              &6.103E-02 & 1.831E-04 & 3.283E-10 & 6.677E-13 & 3.020E-14\\[0.1em]
   &  $ e_{\bm{uv}}^{h,x} $    &9.106E-01 & 1.156E-03 & 1.027E-06 & 1.609E-12 & 2.677E-12\\
\end{tabularx}
{\rule{\temptablewidth}{1pt}}
\end{center}
\end{table}

\vspace{-0.3cm}
\begin{table}[H]
\def\temptablewidth{1\textwidth}
\tabcolsep 0pt	\caption{Errors of the eigenvalue/eigenvector for {\bf Case II} in \textbf{2D}
for \textbf{FM} (upper) and \textbf{Anti-FM} cases (lower) in Example \ref{Exam_Accuracy_Test}.}
\label{Err_2D_1}\small
\begin{center}
\vspace{-1.5em}
{\rule{\temptablewidth}{1pt}}
\begin{tabularx}{\temptablewidth}{@{\extracolsep{\fill}}c|cccccccc}
    & $ $ & $h_0=1$& $h_0/2$ & $h_0/4$ & $h_0/8$ &$h_0/16$ \\[0.2em]
\hline
	\rule{0pt}{12pt}
                &$ e_{\og_x}^h $  				&6.022E-03 & 1.685E-05 & 7.811E-11 & 7.896E-13 & 5.014E-13\\[0.1em]
\multirow{2}{*}{ \textbf{~FM~}}   & $e_{\og_y}^h$  				&6.022E-03 & 1.685E-05 & 7.812E-11 & 8.007E-13 & 4.925E-13\\[0.4em]
   & $ e_{\bm{uv}}^{h,x}$        &3.503E-02 & 8.927E-04 & 4.250E-07 & 4.627E-08 & 2.751E-10\\[0.1em]
                &$ e_{\bm{uv}}^{h,y} $       &3.506E-02 & 8.928E-04 & 4.250E-07 & 2.625E-08 & 2.751E-10\\ [0.4em]
\hline
	\rule{0pt}{12pt}
                &$ e_{\og_x}^h $            		 &6.825E-03 & 1.907E-05 & 4.664E-10 & 1.281E-12 & 4.241E-13  \\[0.1em]
\multirow{2}{*}{ \textbf{~Anit-FM~~} }  &$ e_{\og_y}^h $  				 &6.825E-03 & 1.907E-05 & 2.822E-11 & 1.789E-11 & 4.607E-13  \\[0.4em]
 &$ e_{\bm{uv}}^{h,x} $        &5.013E-02 & 7.403E-04 & 1.558E-07 & 4.135E-10 & 1.147E-12  \\[0.1em]
                & $ e_{\bm{uv}}^{h,_y} $       &5.006E-02 & 7.403E-04 & 1.555E-07 & 5.368E-10 & 1.416E-12  \\
\end{tabularx}
{\rule{\temptablewidth}{1pt}}
\end{center}
\end{table}

\vspace{-0.5cm}
\begin{table}[H]
\def\temptablewidth{1\textwidth}
\tabcolsep 0pt	\caption{Errors of the eigenvalue/eigenvector for {\bf Case III} in \textbf{2D}
for \textbf{FM} (upper) and \textbf{Anti-FM} cases (lower) in Example \ref{Exam_Accuracy_Test}.}
\label{Err_2D_2}\small
\begin{center}
\vspace{-1.5em}
{\rule{\temptablewidth}{1pt}}
\begin{tabularx}{\temptablewidth}{@{\extracolsep{\fill}}c|ccccccccc}
& $ $ & $h_0=1$& $h_0/2$ & $h_0/4$ & $h_0/8$ &$h_0/16$ \\[0.2em]
\hline
	\rule{0pt}{12pt}
                &$e_{\og_x}^h $                 &9.823E-03 & 1.153E-05 & 2.936E-11 & 2.138E-10 & 3.044E-12 \\[0.1em]
\multirow{2}{*}{ \textbf{~FM~}}      &$e_{\og_y}^h $ 			&2.185E-02 & 1.514E-03 & 7.471E-08 & 6.092E-10 & 1.598E-12 \\[0.4em]
	     &$ e_{\bm{uv}}^{h,x}$           &2.644E-02 & 8.649E-04 & 6.728E-07 & 1.398E-08 & 6.398E-09 \\[0.1em]
                &$ e_{\bm{uv}}^{h,y}$  &1.077~~~~~~  & 9.706E-03 & 3.186E-05 & 4.781E-08 & 3.306E-09 \\ [0.4em]
\hline
	\rule{0pt}{12pt}
                &$ e_{\og_x}^h $                &7.222E-03&  3.431E-05 &1.844E-10 &3.888E-13 & 8.724E-13 \\[0.1em]
\multirow{2}{*}{ \textbf{~Anit-FM~~} }     &$ e_{\og_y}^h $                &4.328E-02&  5.590E-04 &4.039E-08 &1.648E-13 & 4.867E-13 \\ [0.4em]
&$ e_{\bm{uv}}^{h,x} $          &3.484E-02&  8.149E-04 &2.777E-07 & 1.562E-12 & 7.409E-11 \\[0.1em]
                &$ e_{\bm{uv}}^{h,y} $          &1.738E-01&  9.128E-03 &1.706E-05 & 4.258E-11 & 7.208E-12 \\[0.1em]
\end{tabularx}
{\rule{\temptablewidth}{1pt}}
\end{center}
\end{table}

\vspace{-0.5cm}
\begin{table}[H]
\def\temptablewidth{1\textwidth}
\tabcolsep 0pt	\caption{Errors of the eigenvalue/eigenvector for {\bf Case IV} in \textbf{3D}
for \textbf{FM} (upper) and \textbf{Anti-FM} cases (lower) in Example \ref{Exam_Accuracy_Test}.}
\label{Err_3D}\small
\begin{center}
\vspace{-1.5em}
{\rule{\temptablewidth}{1pt}}
\begin{tabularx}{\temptablewidth}{@{\extracolsep{\fill}}c|ccccccc}
                & $ $                               &$h_0=1$   & $h_0/2$   & $h_0/4$   & $h_0/8$ \\[0.2em]
\hline
	\rule{0pt}{12pt}
                &$ e_{\og_x}^h $         		   &3.526E-03 & 6.684E-06 & 6.108E-11 & 1.027E-11 \\  [0.1em]
                &$ e_{\og_y}^h $  			   &3.526E-03 & 6.684E-06 & 7.500E-12 & 1.028E-11 \\  [0.1em]
\multirow{2}{*}{ \textbf{FM}}         &$ e_{\og_z}^h $  			   &3.526E-03 & 6.684E-06 & 4.598E-11 & 1.028E-11 \\  [0.4em]
  &$ e_{\bm{uv}}^{h,x} $          &4.434E-02 & 6.697E-04 & 1.116E-07 & 4.046E-09 \\  [0.1em]
                &$ e_{\bm{uv}}^{h,y} $          &4.433E-02 & 6.697E-04 & 1.116E-07 & 4.046E-09 \\  [0.1em]
                &$ e_{\bm{uv}}^{h,z} $          &4.443E-02 & 6.697E-04 & 1.116E-07 & 4.046E-09 \\  [0.4em]
\hline
	\rule{0pt}{12pt}
                &$e_{\og_x}^h $         		   &2.061E-03 & 3.620E-06 & 1.602E-11 & 8.528E-13 \\  [0.1em]
                &$ e_{\og_y}^h $         		   &2.061E-03 & 3.620E-06 & 1.213E-12 & 5.934E-13 \\  [0.1em]
\multirow{2}{*}{ \textbf{~Anit-FM~~} }     &$ e_{\og_z}^h $         		   &2.061E-03 & 3.620E-06 & 3.218E-11 & 5.874E-13 \\  [0.4em]
  &$ e_{\bm{uv}}^{h,x} $          &5.726E-02 & 5.812E-04 & 5.043E-08 & 5.873E-11 \\  [0.1em]
                &$ e_{\bm{uv}}^{h,y} $          &5.712E-02 & 5.812E-04 & 5.042E-08 & 5.873E-11 \\  [0.1em]
                &$ e_{\bm{uv}}^{h,z} $           &5.713E-02 & 5.812E-04 & 5.043E-08 & 5.873E-11 \\  [0.1em]
\end{tabularx}
{\rule{\temptablewidth}{1pt}}
\end{center}
\end{table}

\vspace{-0.2cm}
\begin{figure}[H]
\centering
  \includegraphics[width=8.0cm,height=6.1cm]{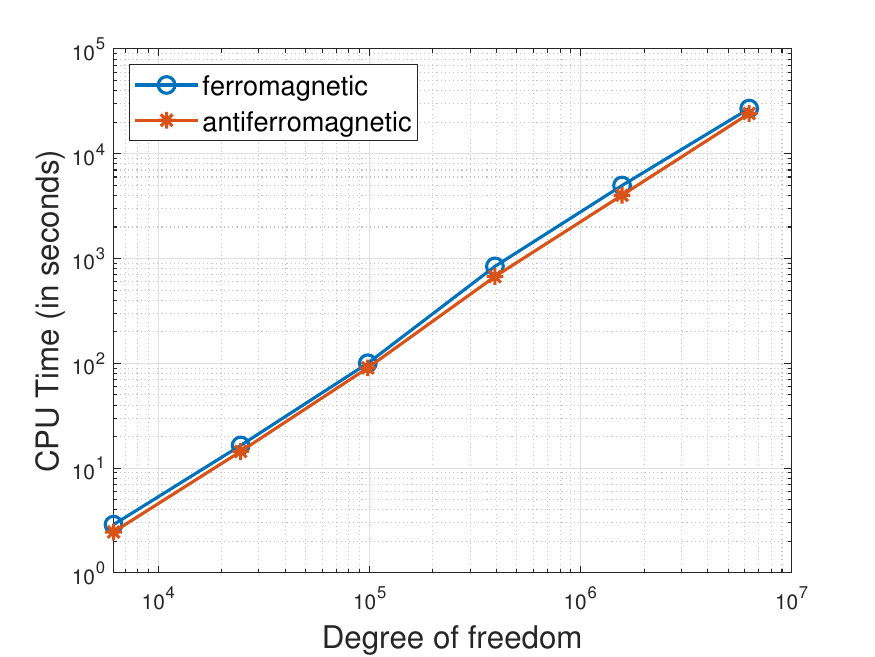}
  \includegraphics[width=8.0cm,height=6.1cm]{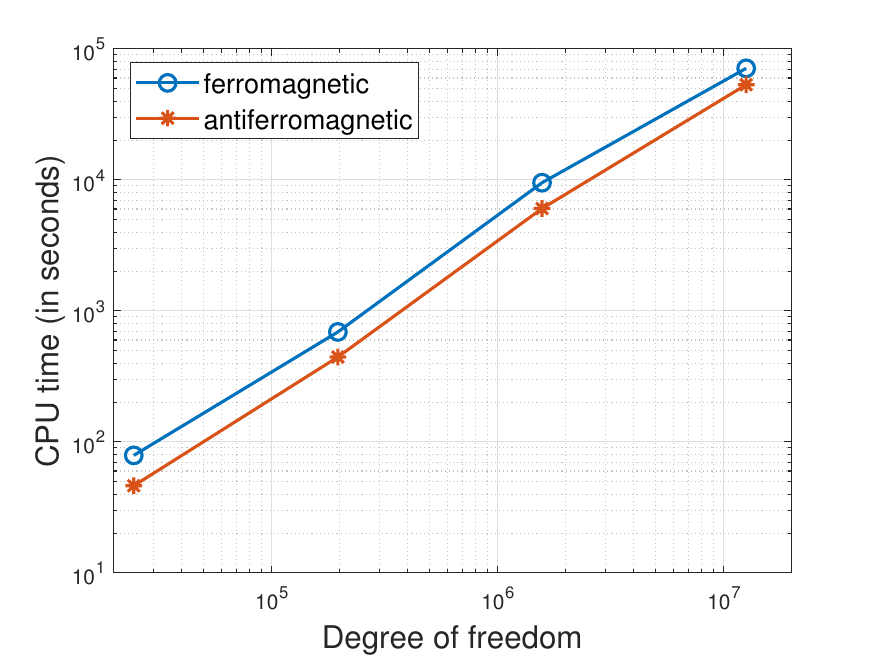}
\caption{The computational times for {\bf Case I} in {\bf 2D} (left) and  {\bf Case II} in {\bf 3D} (right) in Example \ref{Exam_Efficiency_Test}.}
\label{times_efficiency_test}
\end{figure}
\vspace{-1.5em}
\begin{exam}[\bf Efficiency]\label{Exam_Efficiency_Test}
We investigate the efficiency performance using different mesh sizes for ferromagnetic and antiferromagnetic cases in 2D and 3D.
To this end, we consider the following cases
\begin{itemize}
\item[] \hspace{-0.5cm}{\bf Case I.}  Isotropic 2D case: $\gamma_x = \gamma_y=1 $.
\item[] \hspace{-0.5cm}{\bf Case II.} Isotropic 3D case: $\gamma_x=\gamma_y=\gamma_z=1 $.
\end{itemize}

\end{exam}

Figure \ref{times_efficiency_test} presents the computation times (measured in seconds) versus the degree of freedom for 2D and 3D cases.
The degree of freedom varies ${\tt DOF} = 6\,144,~24\,576,~98\,304,~393\,216,~1\,572\,864,$\\ $6\,291\,456$ in 2D,
and ${\tt DOF} =~24\,576,~196\,608,~\,1\,572\,864,~12\,582\,912$ in 3D, while the number of eigenvalue is kept unchanged as ${\tt nev} = 40$.
It is obvious that the degree of freedom is much larger than the number of eigenvalues, i.e., ${\tt DOF}\gg{\tt nev}$, therefore, the overall computational complexity is around $\mathcal{O}({\tt DOF}\log ({\tt DOF}))$, which can be observed quite clear from Figure \ref{times_efficiency_test}.
Given the fact that the discrete BdG system is nonsymmetric and dense and the first ${\tt nev} = 40$ eigenpairs of 3D problem can be computed with a {\bf 12 million} degree of freedom {\bf within hours} using only solely sequentially one-CPU computing, it is reasonable to expect a promising future in physical applications.

\subsection{Applications}
\begin{figure}[H]
\centering
  \includegraphics[scale=0.35]{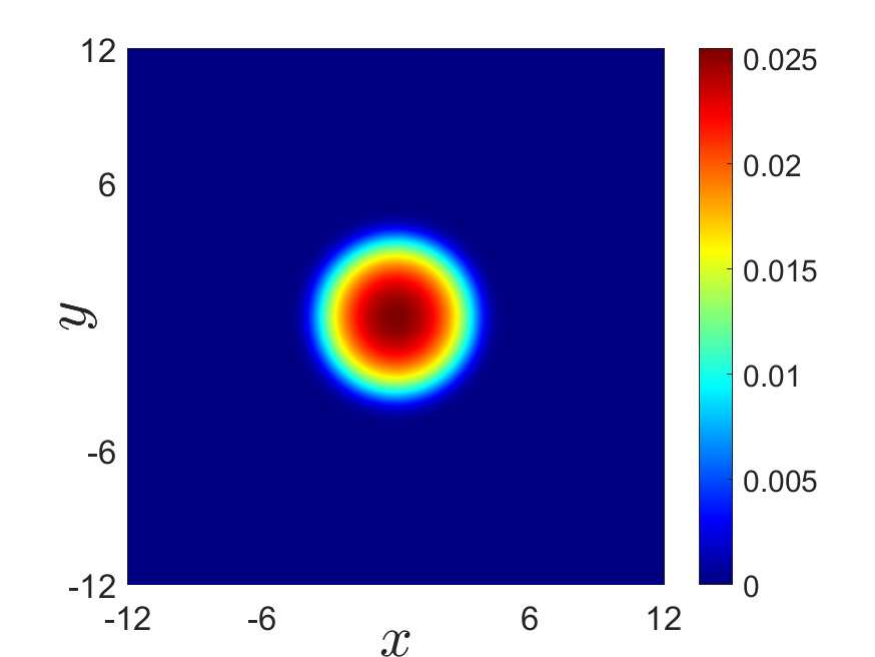}
  \hspace{-0.5cm}
  \includegraphics[scale=0.35]{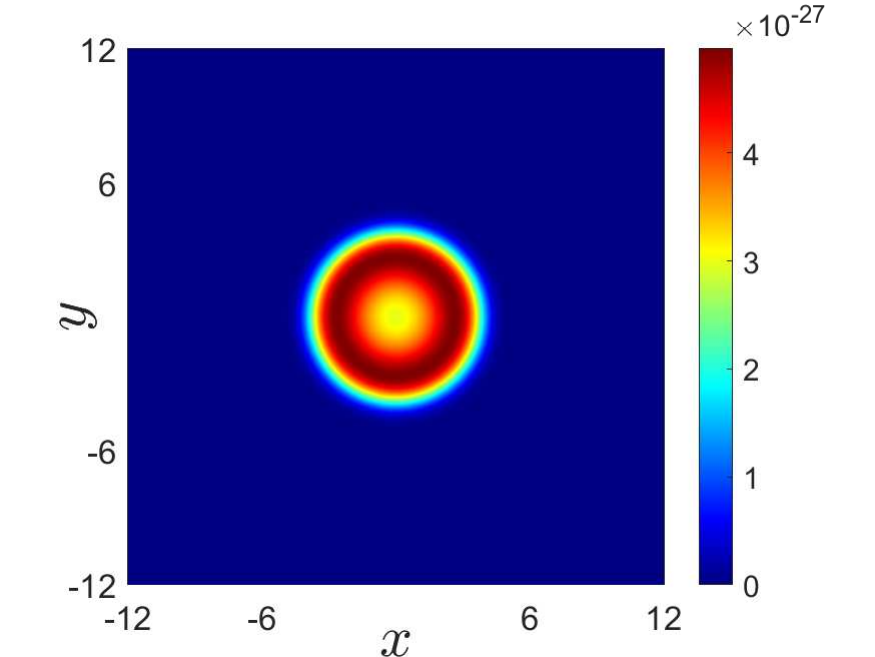}
  \hspace{-0.5cm}
  \includegraphics[scale=0.35]{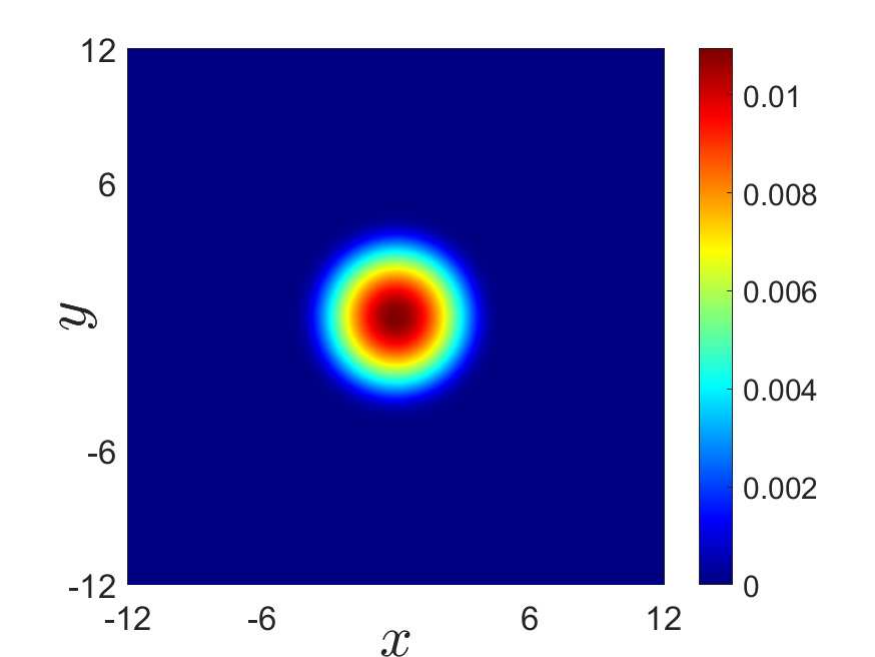}
 \\
  \includegraphics[scale=0.35]{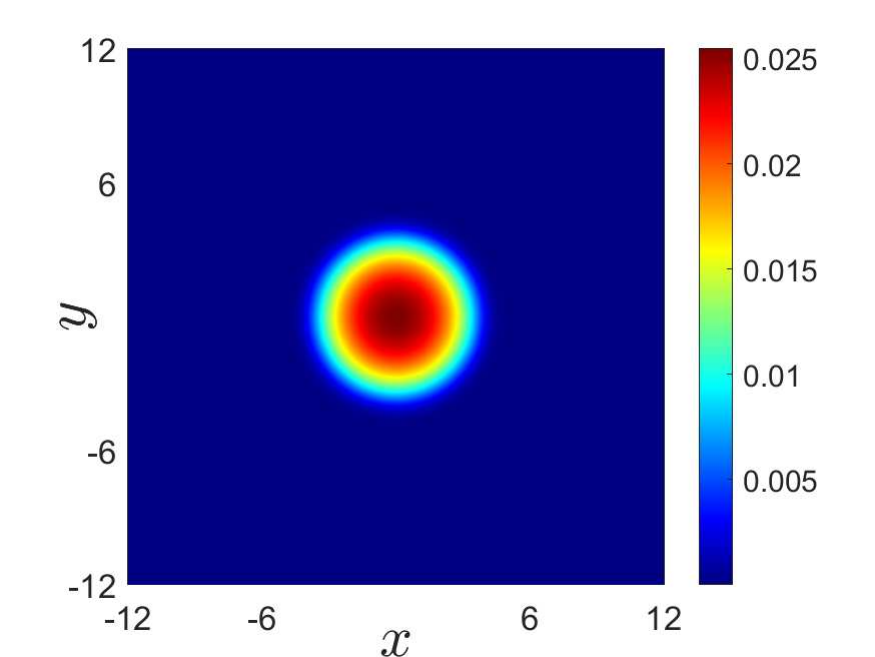}
  \hspace{-0.5cm}
  \includegraphics[scale=0.35]{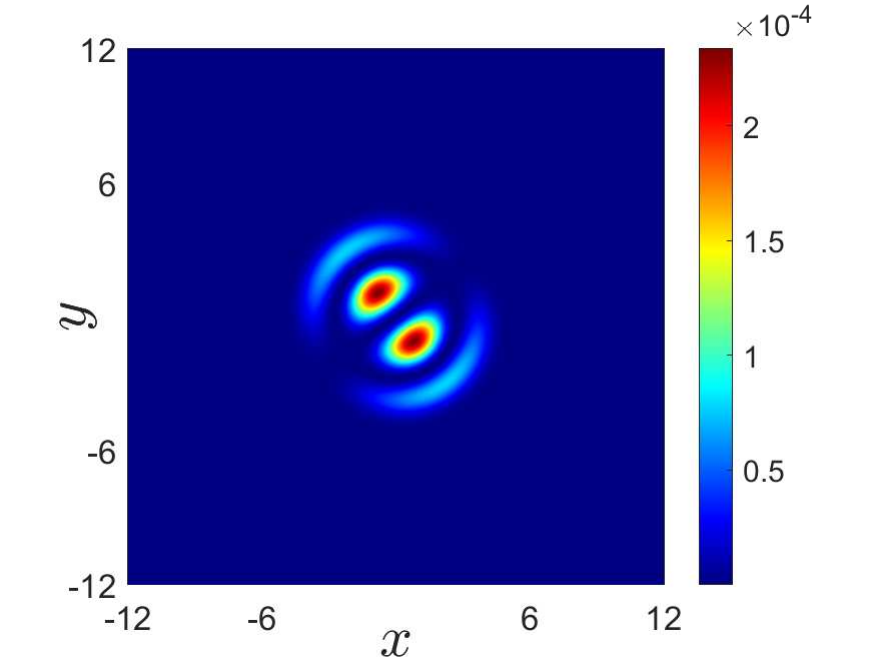}
  \hspace{-0.5cm}
  \includegraphics[scale=0.35]{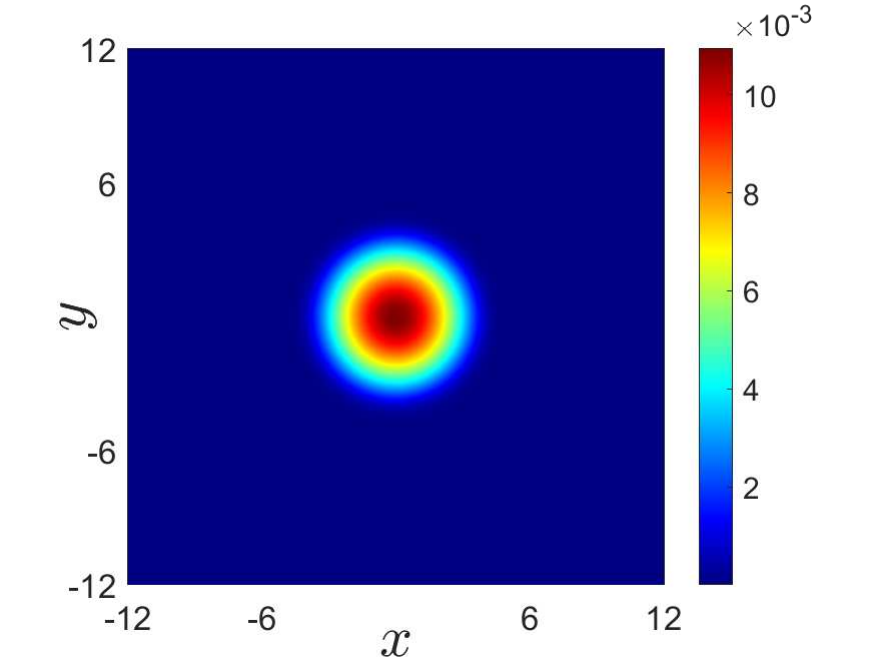}
 \\
  \centering
  \includegraphics[scale=0.35]{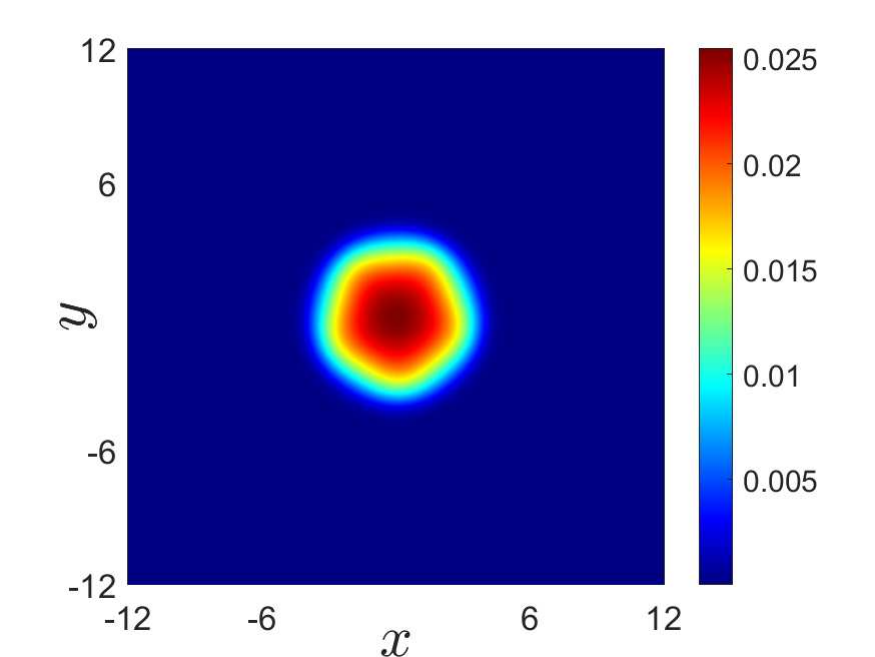}
  \hspace{-0.5cm}
  \includegraphics[scale=0.35]{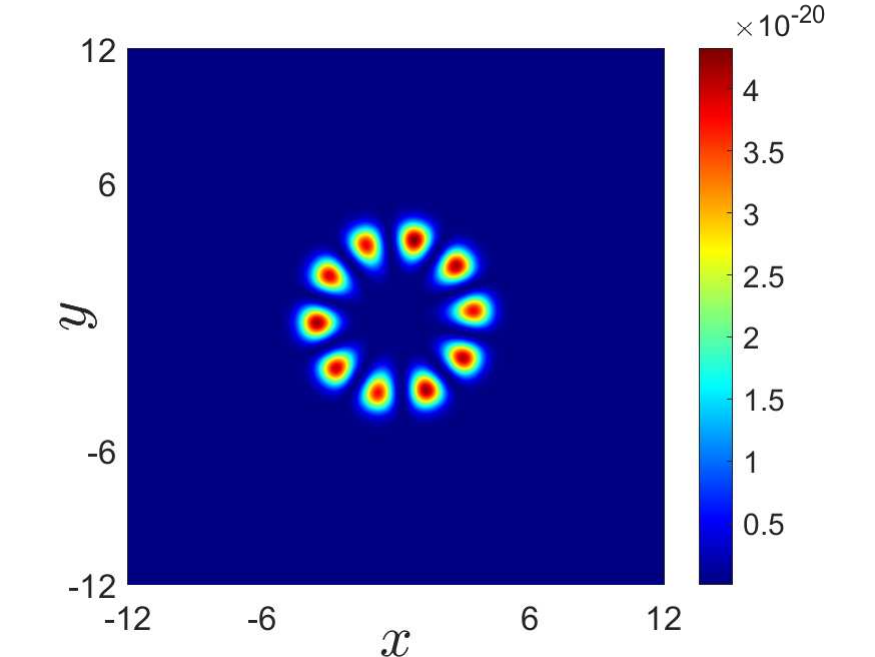}
  \hspace{-0.5cm}
  \includegraphics[scale=0.35]{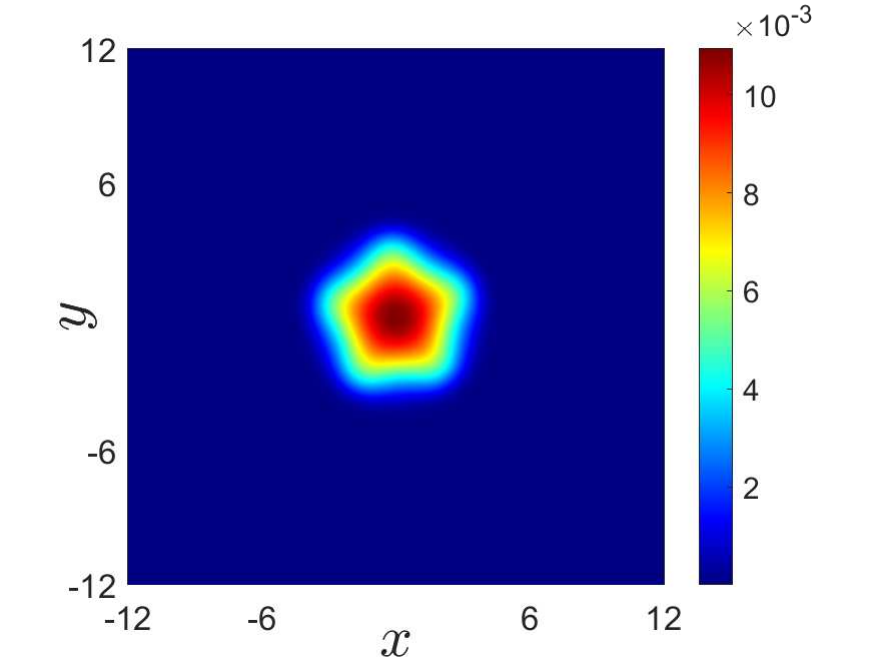}
\caption{Snapshots of $|\phi_j(\bx)|^2$ (top row) and perturbed densities $n_j^\ell(\bx,t=10.6)$
by different excitation modes: $\ell=20,35$ (middle and bottom row) for {\bf Case I} in Example \ref{perturb} ($j=1,0,-1$ from left to right).}
\label{ptb_anti_iso}
\end{figure}

In this section, we apply our solver to investigate the excitation spectrum
and Bogoliubov amplitudes of spin-1 BEC around the ground state.
To visualize a normal mode, similar to \eqref{waveAssump}, we analyze the evolution of the perturbed density profile \cite{YiLowLying18}
\bea
{n}^\ell_j(\bx,t)=\left|\left[\phi_j(\bx)+ \varepsilon\big(u^\ell_j(\bx)
e^{-i\og_\ell t}+\bar{v}_j^\ell(\bx)e^{i\og_\ell t}\big) \right]\right|^2,\ \ j= 0,\pm1,
\eea
which reveals the nature of the excitations with $\ell$ being the eigenpair index.
To this end, we choose a very fine mesh size $h = 1/8$ in both two and three spatial dimensions.

\begin{exam}[\textbf{2D Case}]
\label{perturb}
We investigate the perturbed density by different excitation modes in ferromagnetic and antiferromagnetic condensates in 2D.
To this end, we choose $\varepsilon=0.1$ and consider the following cases
\begin{itemize}
\item[] \hspace{-0.5cm}{\bf Case I.} Isotropic antiferromagnetic condensates:  $\gamma_x = \gamma_y=1$, $\ell=20,35$.
\item[]\hspace{-0.65cm} {\bf Case II.} Anisotropic ferromagnetic condensates:  $\gamma_x = \gamma_y/2 =1$, $\ell=30,55$.
\end{itemize}
\end{exam}
Figure \ref{ptb_anti_iso}-\ref{ptb_ferr_ani} displays the numerical excitations $n_j^\ell(\bx,t)$ of the BdG equations (at $t = 10.6$) that are associated
with the eigenvalues $\omega_\ell$ for {\bf Case I--II}. From
these figures, we can see that both the spinor and external potential affect
the shape of the excitations essentially and significantly.
The eigenmodes $(\bu^\ell, \bv^\ell)$ are symmetric or antisymmetric in a symmetric or antisymmetric external potential, respectively.
Meanwhile,

\begin{figure}[H]
\centering
  \includegraphics[scale=0.35]{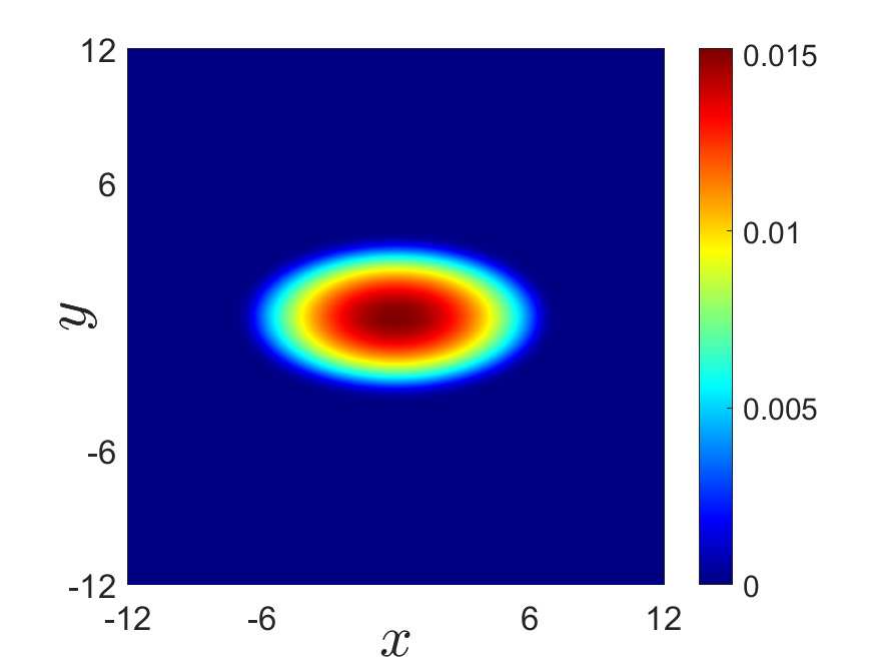}
  \hspace{-0.5cm}
  \includegraphics[scale=0.35]{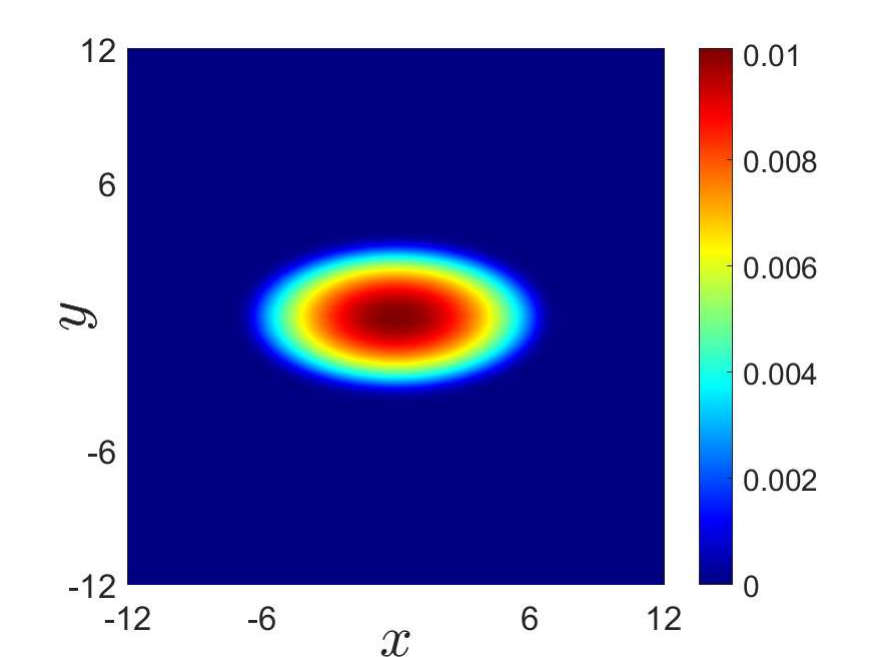}
  \hspace{-0.5cm}
  \includegraphics[scale=0.35]{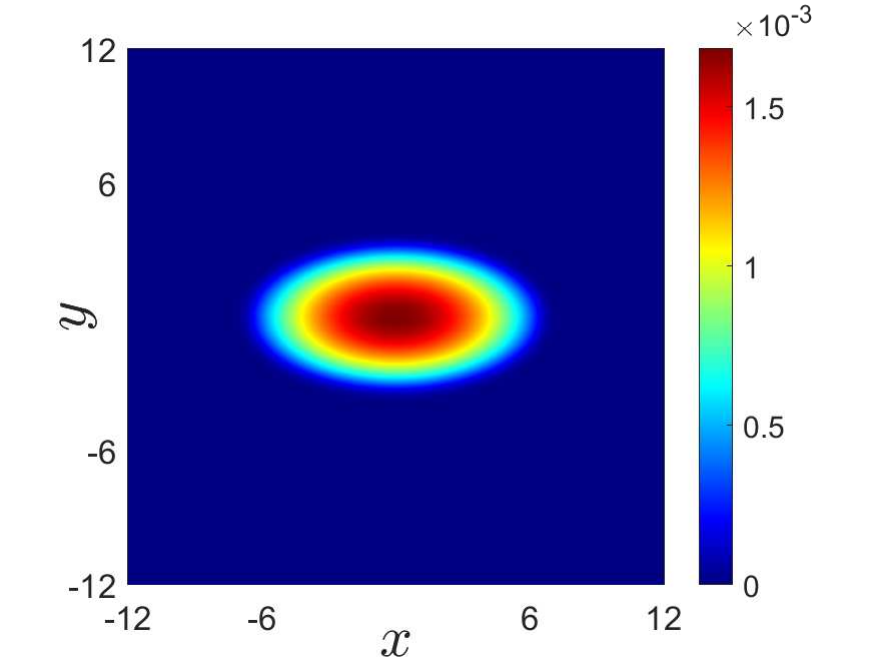}
 \\
  \includegraphics[scale=0.35]{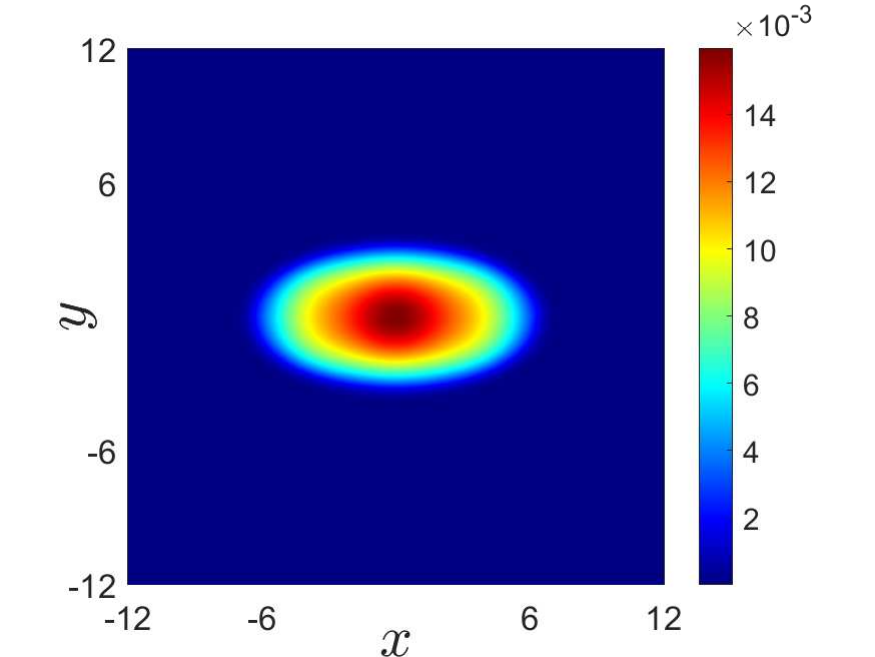}
  \hspace{-0.5cm}
  \includegraphics[scale=0.35]{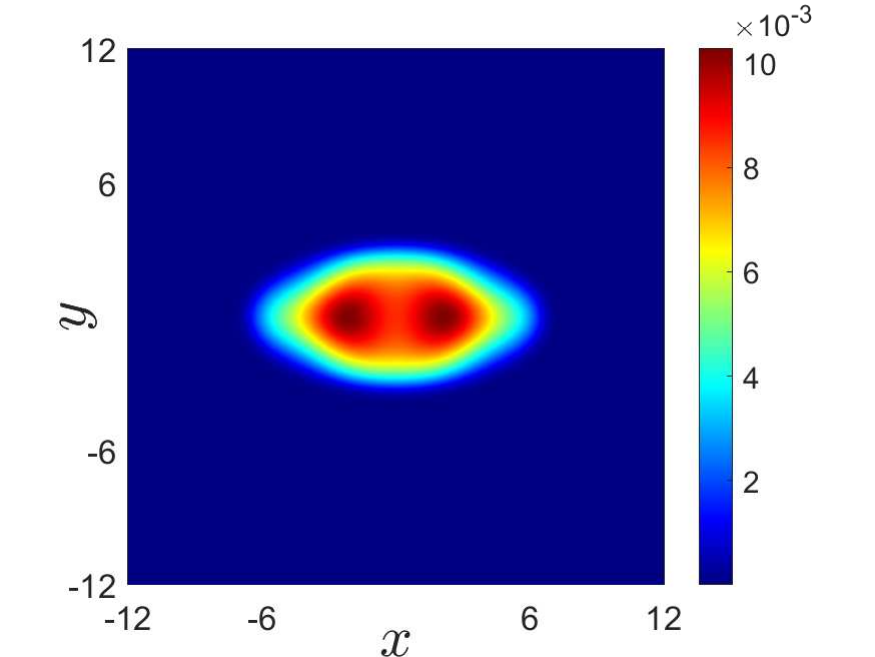}
  \hspace{-0.5cm}
  \includegraphics[scale=0.35]{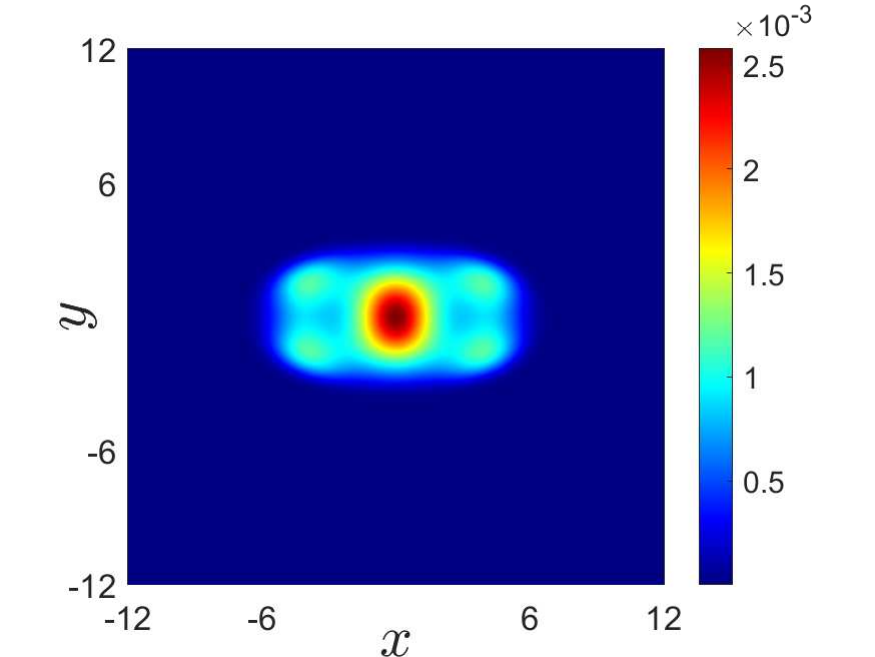}
 \\
  \centering
  \includegraphics[scale=0.35]{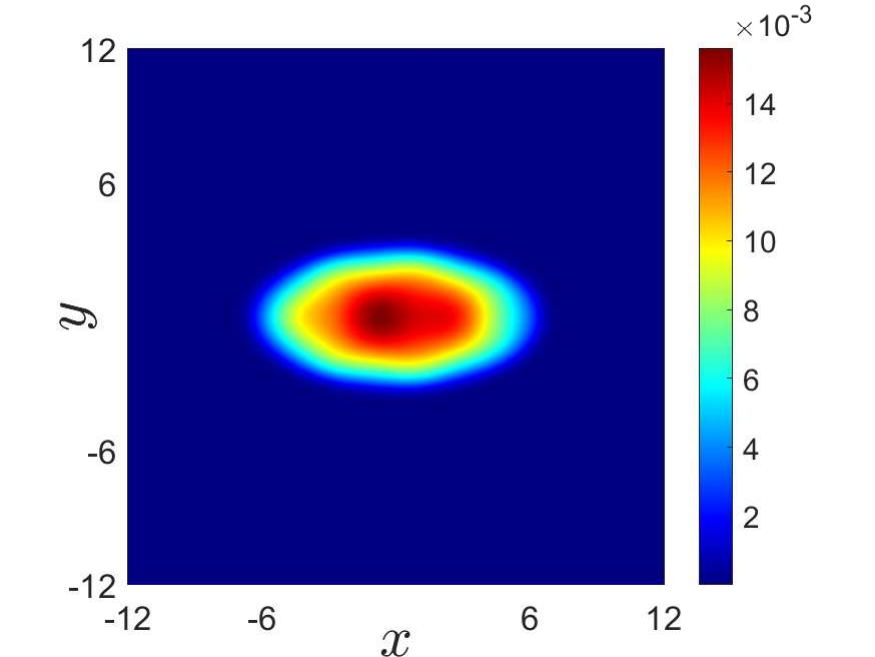}
  \hspace{-0.5cm}
  \includegraphics[scale=0.35]{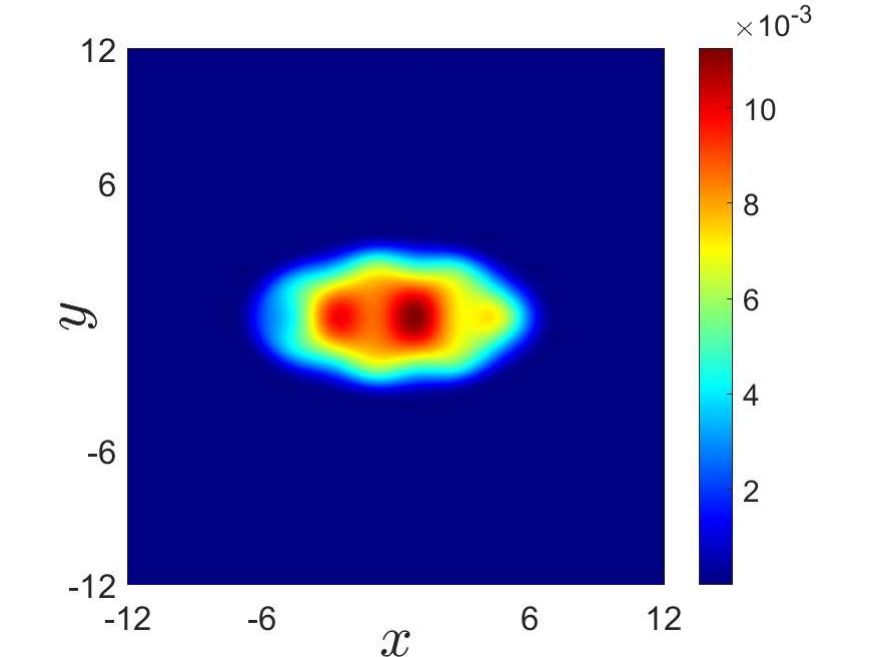}
  \hspace{-0.5cm}
  \includegraphics[scale=0.35]{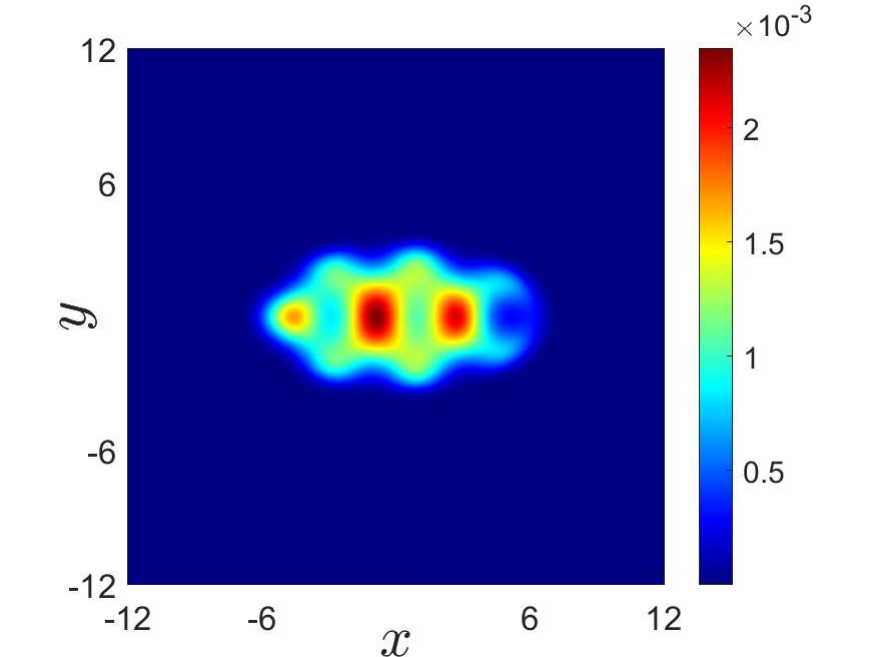}
\caption{Snapshots of $|\phi_j(\bx)|^2$ (top row) and perturbed densities $n_j^\ell(\bx,t=10.6)$
by different excitation modes: $\ell=30,55$ (middle and bottom row) for {\bf Case II} in Example \ref{perturb} ($j=1,0,-1$ from left to right).}
\label{ptb_ferr_ani}
\end{figure}


\noindent  the eigenmodes will be compressed along the direction with a larger trapping frequency.
Indeed, the presence of spinor and anisotropic external potential brings many more rich phase diagrams for eigenmodes of the BdG equations,
which will be detailed in the future.

\begin{exam}[\textbf{3D Case}]
\label{amplitudes}
We study the Bogoliubov amplitudes of the BdG equations in ferromagnetic and antiferromagnetic condensates in 3D.
To this end, we study the following four cases
\begin{itemize}
\item[] \hspace{-0.5cm}{\bf Case I.} Isotropic ferromagnetic condensates:  $\gamma_x = \gamma_y=\gamma_z=1$, $\ell=25$.
\item[] \hspace{-0.5cm}{\bf Case II.} Isotropic antiferromagnetic condensates: $\gamma_x = \gamma_y=\gamma_z=1$, $\ell=26$.
\item[]\hspace{-0.65cm} {\bf Case III.} Anisotropic ferromagnetic condensates: $\gamma_x = \gamma_y =\gamma_z/2=1$, $\ell=62$.
\item[] \hspace{-0.5cm}{\bf Case IV.}  Anisotropic antiferromagnetic condensates: $\gamma_x = \gamma_y=\gamma_z/2=1$, $\ell=42$.
\end{itemize}
\end{exam}
\begin{figure}[H]
\centering
  \includegraphics[scale=0.35]{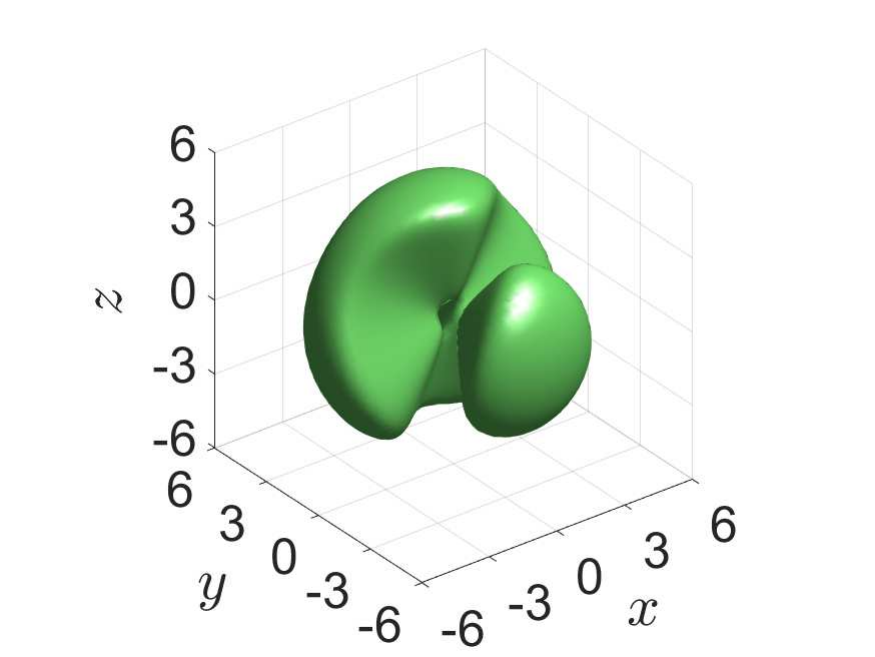}
  \hspace{-0.9cm}
  \includegraphics[scale=0.35]{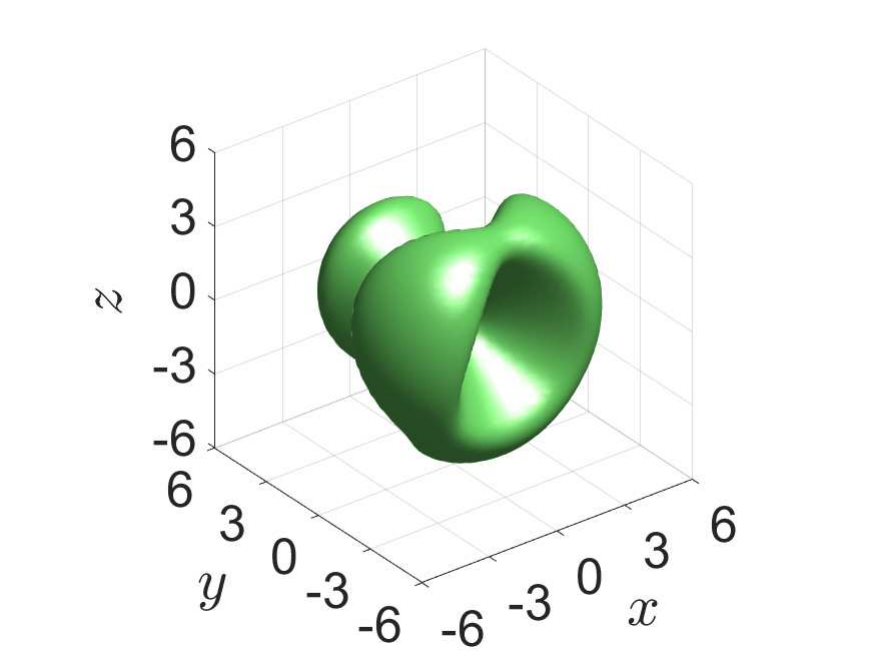}
  \hspace{-0.9cm}
  \includegraphics[scale=0.35]{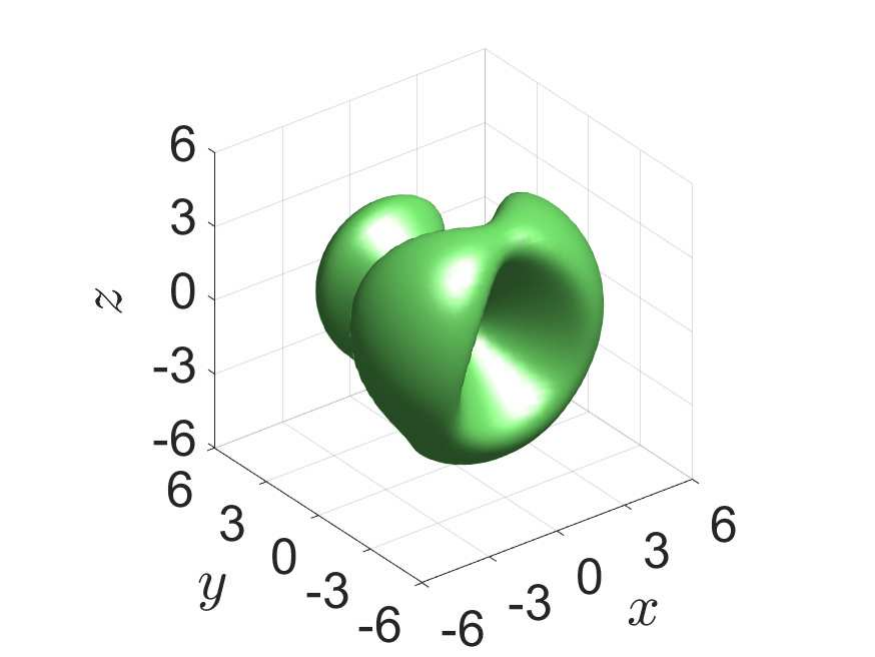}
 \\
  \centering
  \includegraphics[scale=0.35]{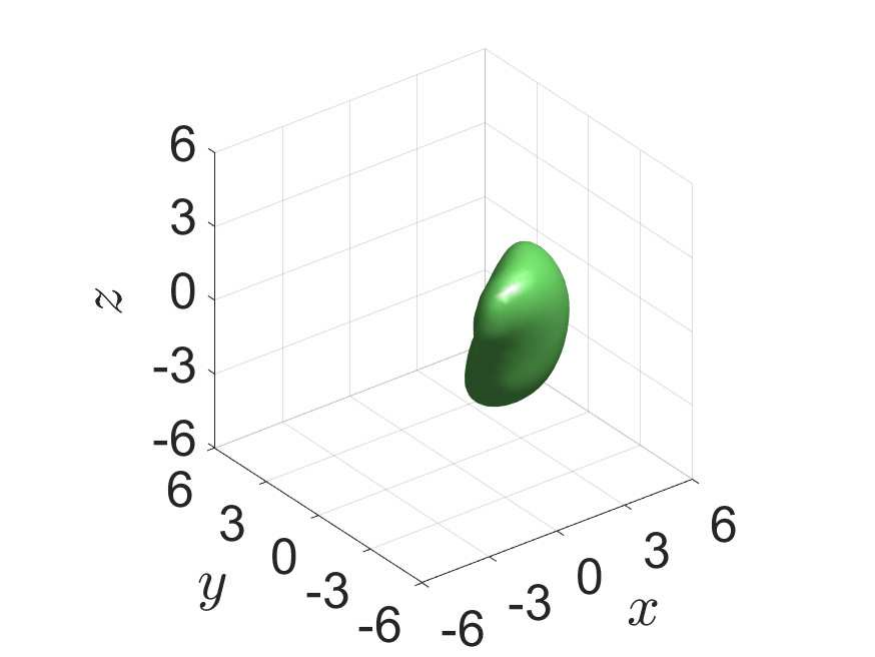}
  \hspace{-0.9cm}
  \includegraphics[scale=0.35]{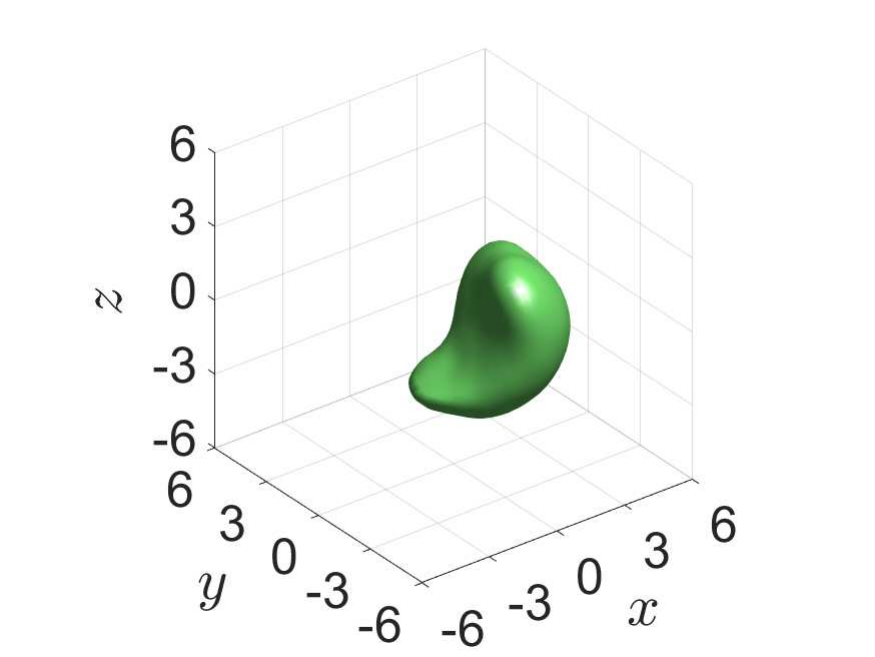}
  \hspace{-0.9cm}
  \includegraphics[scale=0.35]{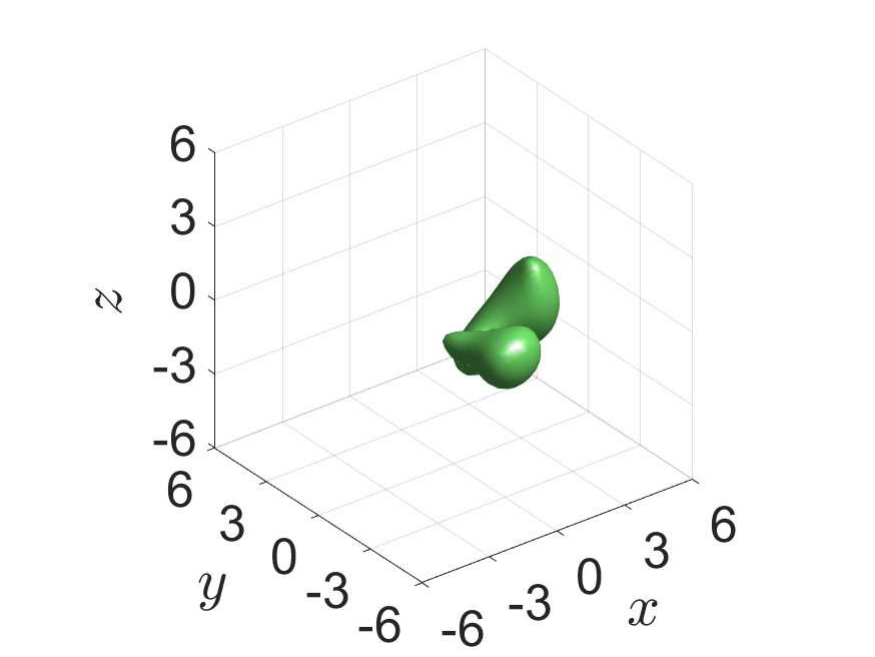}
 \\
  \centering
  \includegraphics[scale=0.35]{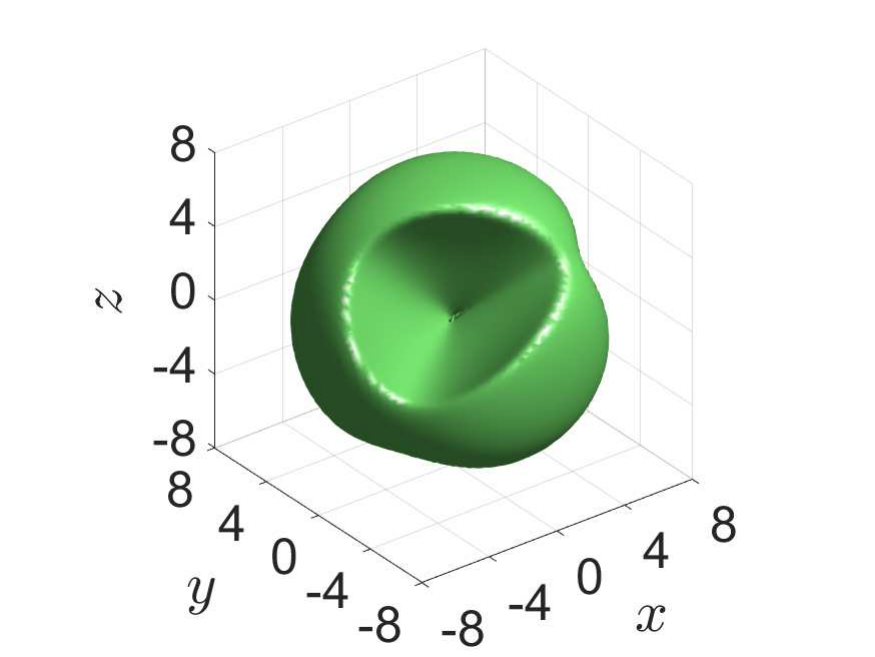}
  \hspace{-0.9cm}
  \includegraphics[scale=0.35]{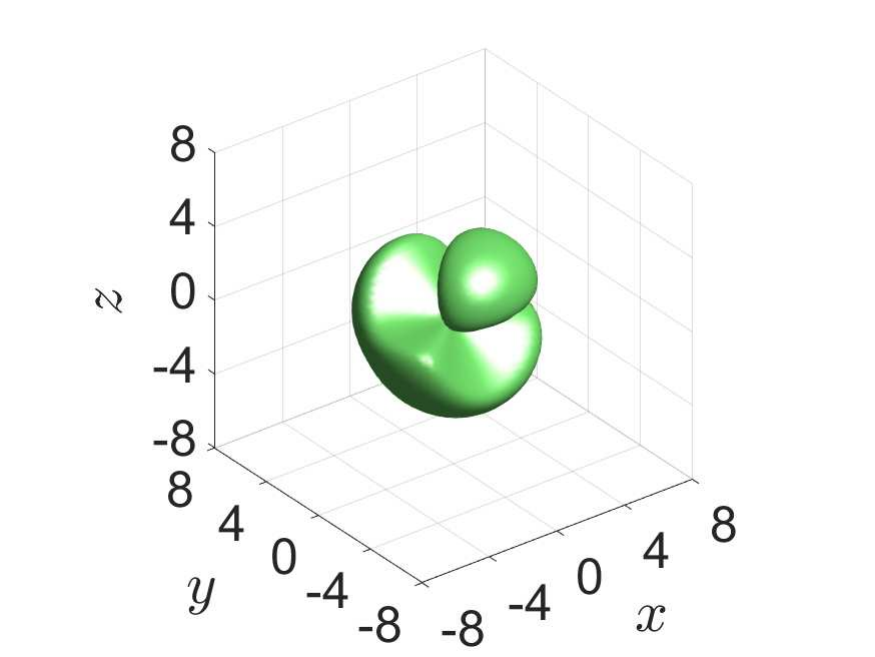}
  \hspace{-0.9cm}
  \includegraphics[scale=0.35]{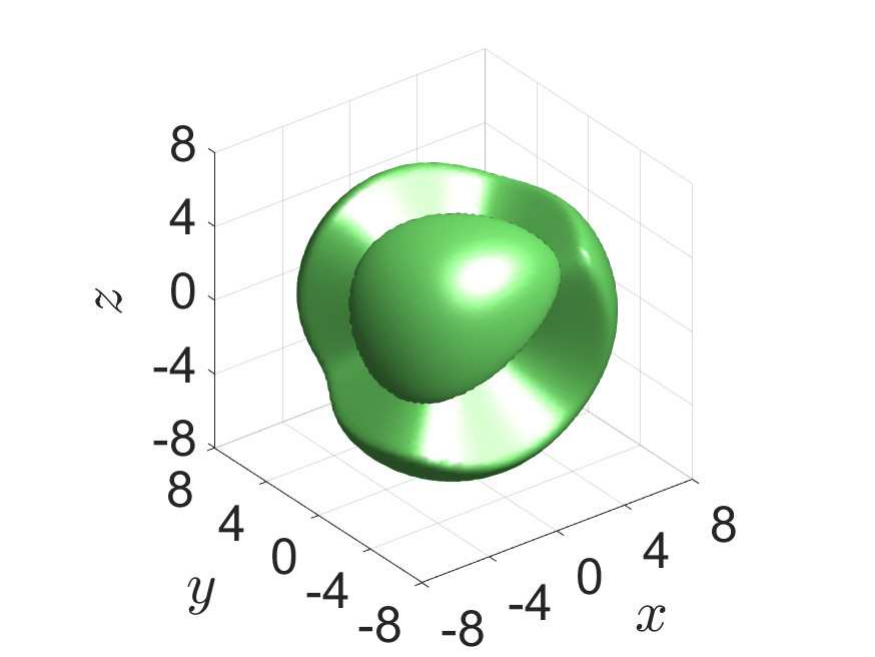}
 \\
  \centering
  \includegraphics[scale=0.35]{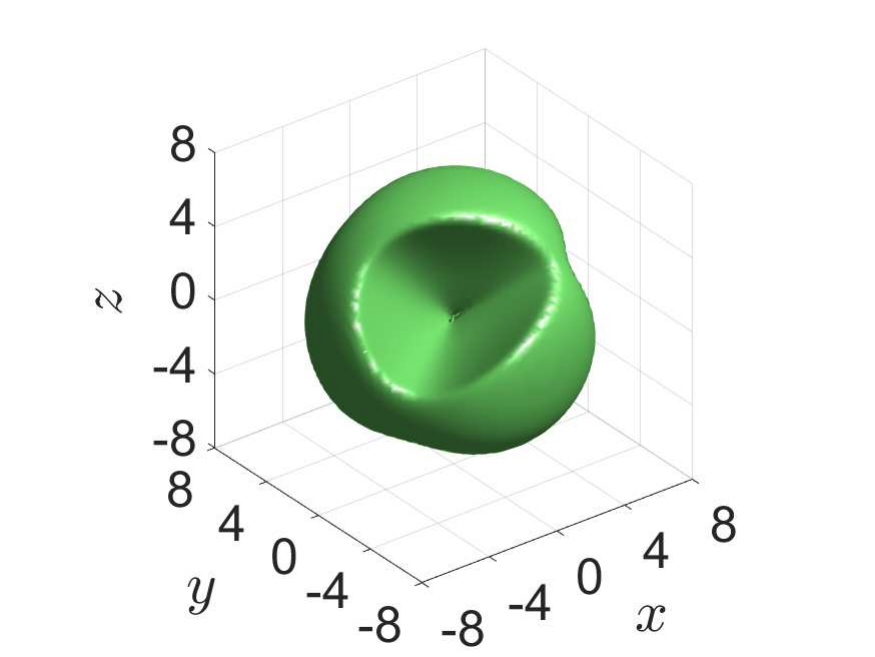}
  \hspace{-0.9cm}
  \includegraphics[scale=0.35]{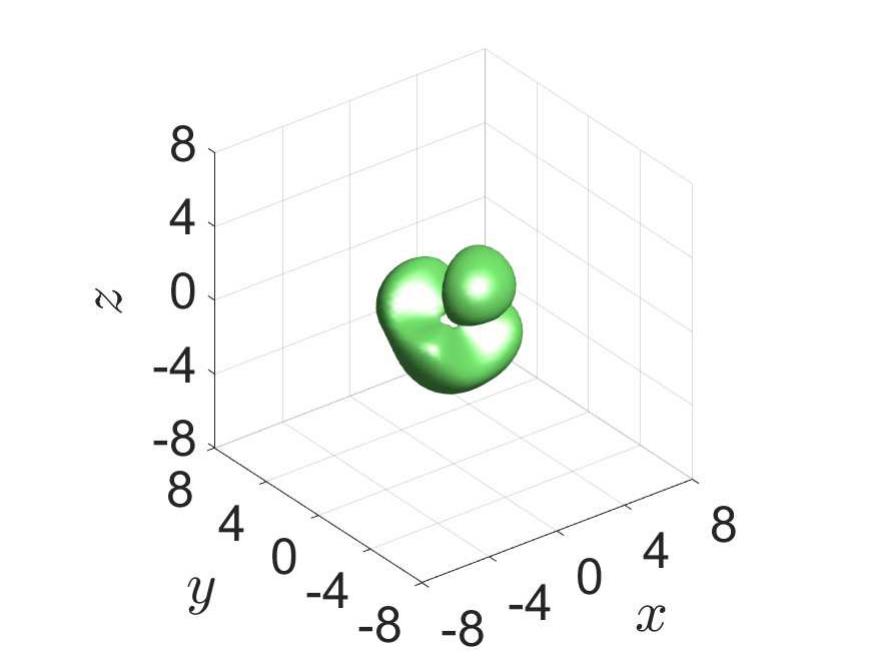}
  \hspace{-0.9cm}
  \includegraphics[scale=0.35]{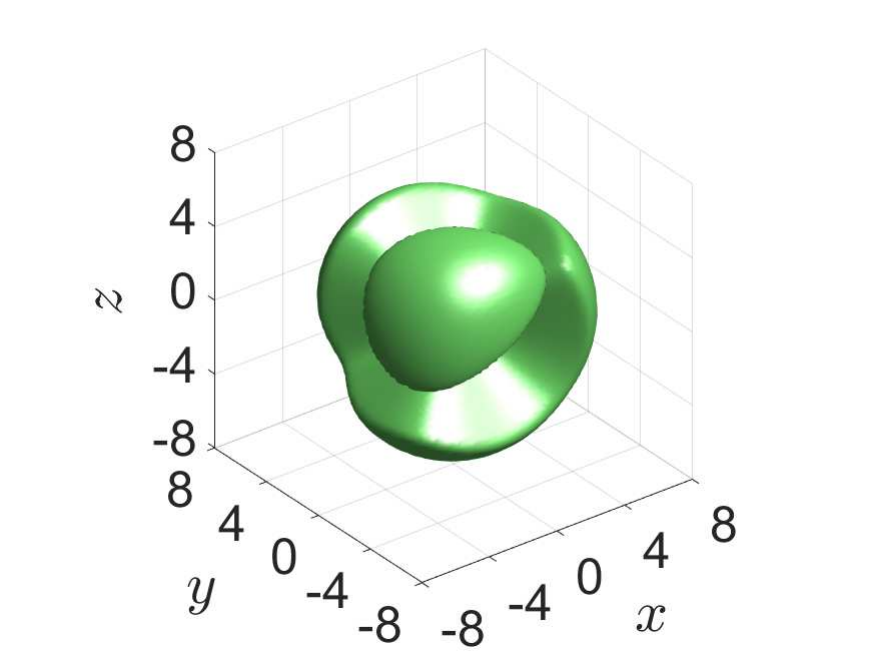}
\caption{
In Example \ref{amplitudes}, isosurface plots of the Bogoliubov amplitudes for {\bf Case I}: $u_j^{25}=10^{-3}$ (1st row), $v_j^{25}=10^{-11}$ (2nd row),
and {\bf Case II}: $u_j^{26}=10^{-10}$ (3rd row), $v_j^{26}=10^{-11}$ (4th row) (from left to right: $j=1,0,-1$).}
\label{eigv_iso}
\end{figure}

Figure \ref{eigv_iso}--\ref{eigv_ani} display isosurface plots of the eigenmodes
$\bu^\ell = (u^\ell_1,u^\ell_0,u^\ell_{-1})^\top$ and $\bv^\ell = (v^\ell_1,v^\ell_0,v^\ell_{-1})^\top$ associated with the different $\ell$ for {\bf Case I--IV}.
From these figures, we can see that the external potential affects the shape of the eigenmodes $\bu^\ell$ and $\bv^\ell$ essentially and significantly.
The eigenmodes can be symmetric or antisymmetric even in a symmetric trapping potential.
Meanwhile, all eigenmodes will be compressed along the direction with a larger trapping frequency.
For the 3D case, the presence of spinor and anisotropic external potential brings in many rich phase diagrams for eigenmodes, and we shall leave them as a future study.

\begin{figure}[H]
\centering
  \includegraphics[scale=0.35]{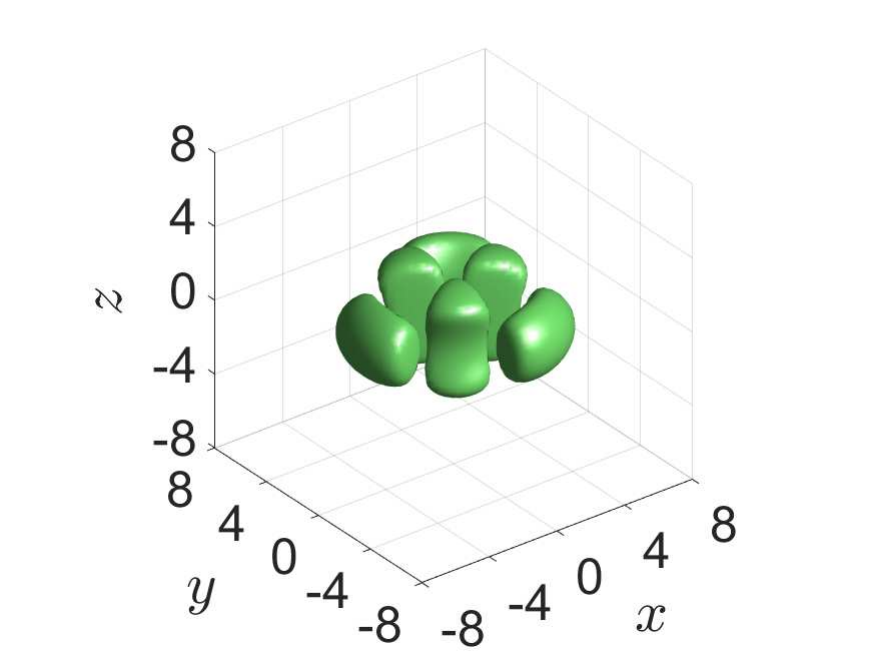}
  \hspace{-0.9cm}
  \includegraphics[scale=0.35]{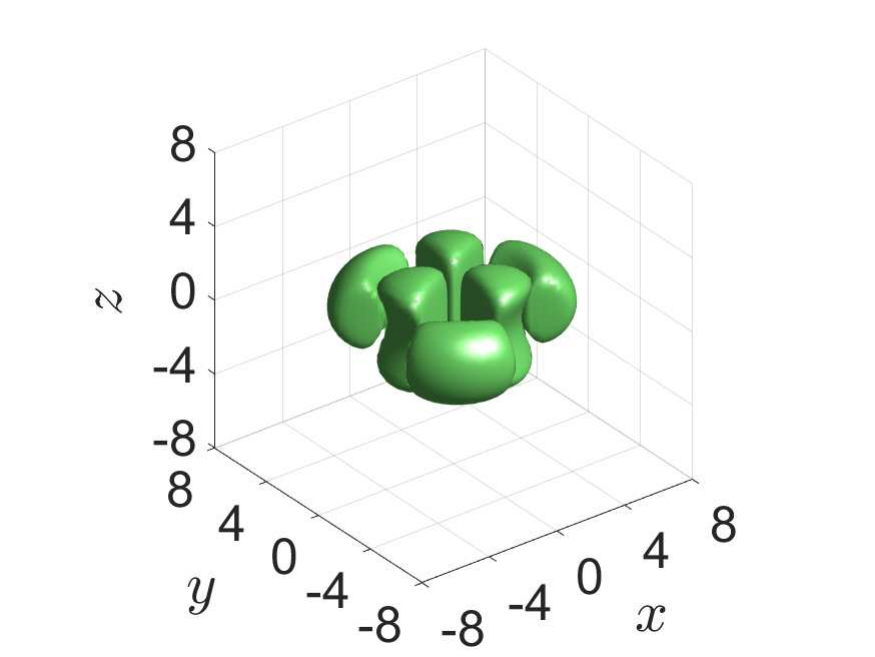}
  \hspace{-0.9cm}
  \includegraphics[scale=0.35]{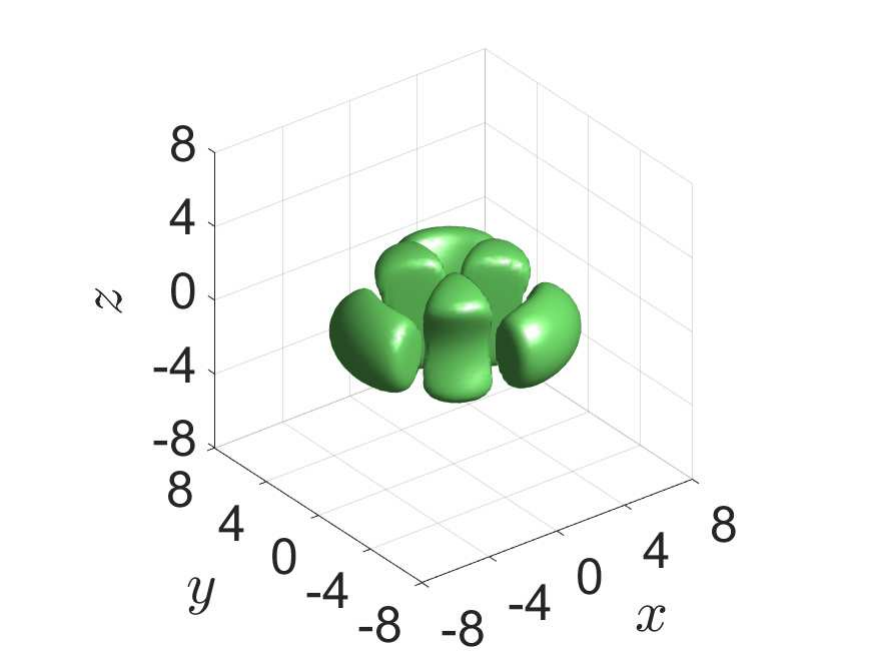}
 \\
  \centering
  \includegraphics[scale=0.35]{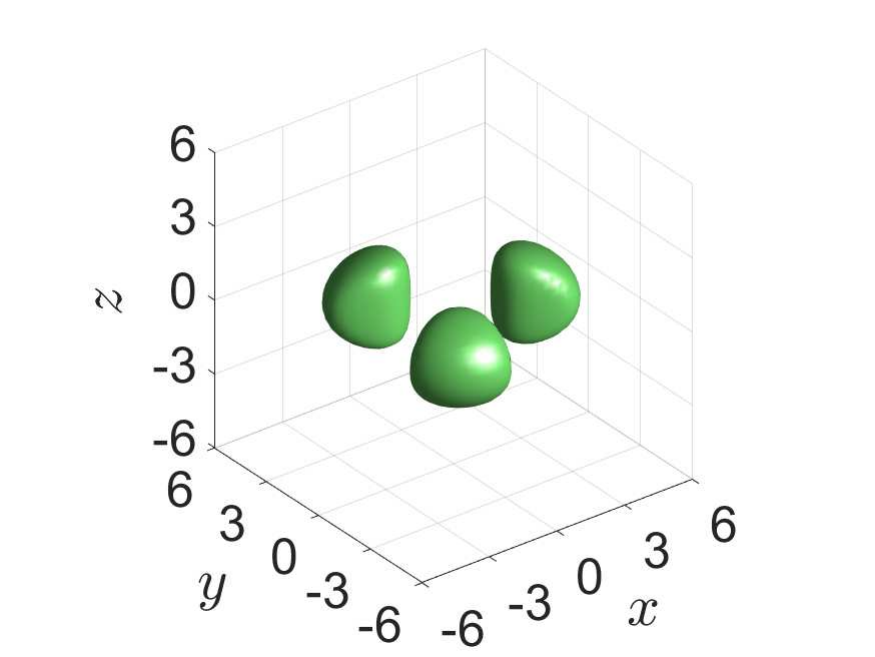}
  \hspace{-0.9cm}
  \includegraphics[scale=0.35]{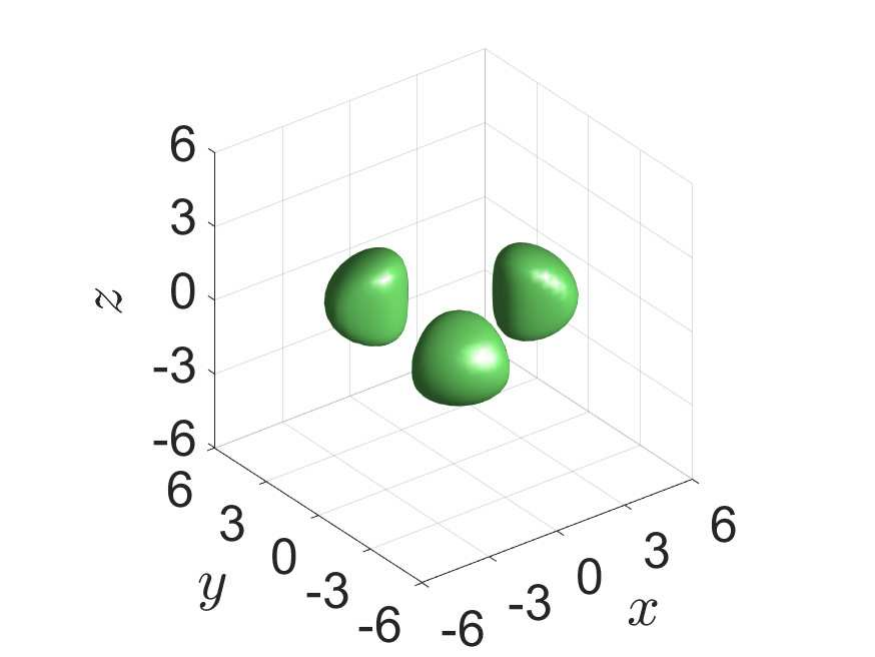}
  \hspace{-0.9cm}
  \includegraphics[scale=0.35]{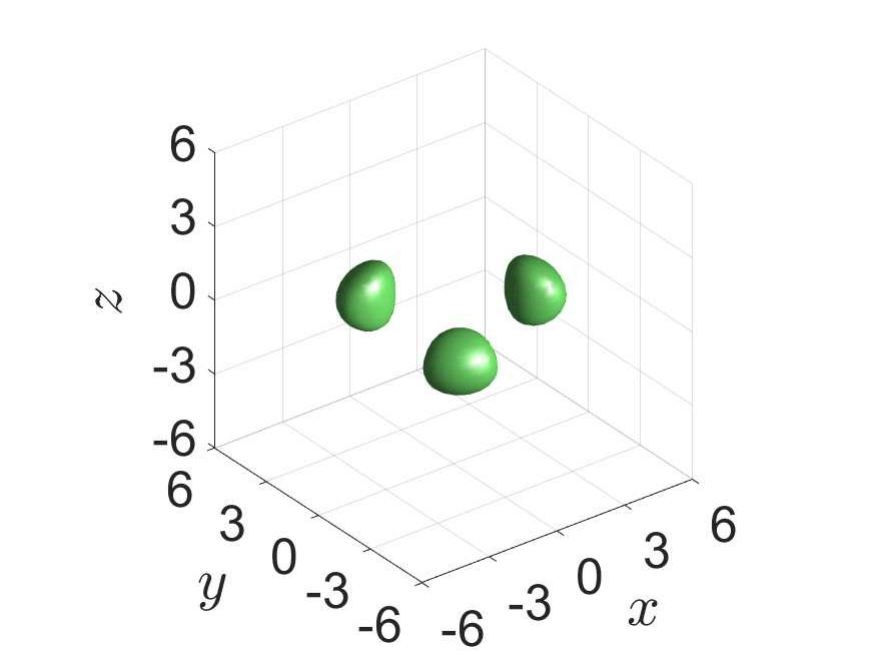}
 \\
  \centering
  \includegraphics[scale=0.35]{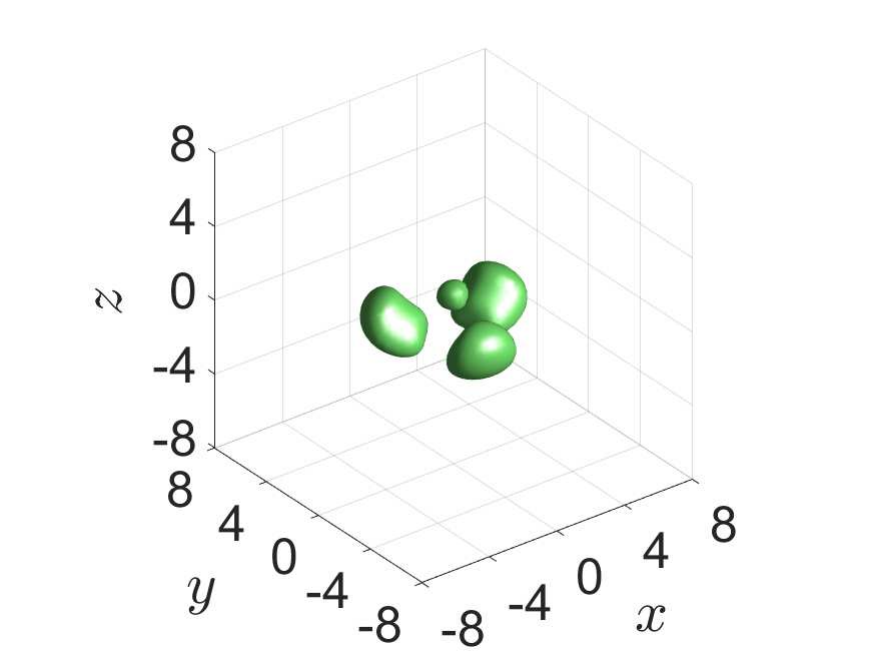}
  \hspace{-0.9cm}
  \includegraphics[scale=0.35]{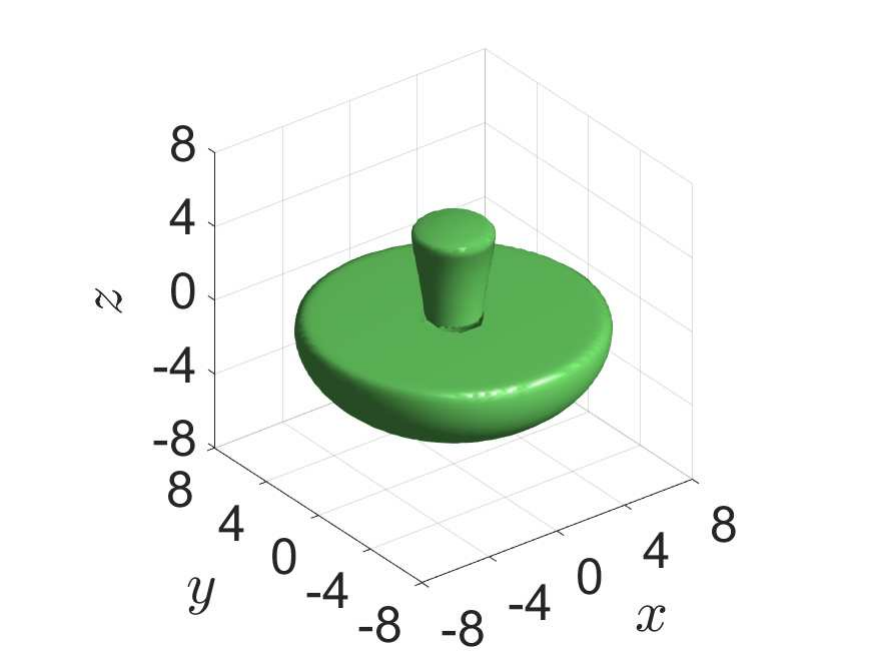}
  \hspace{-0.9cm}
  \includegraphics[scale=0.35]{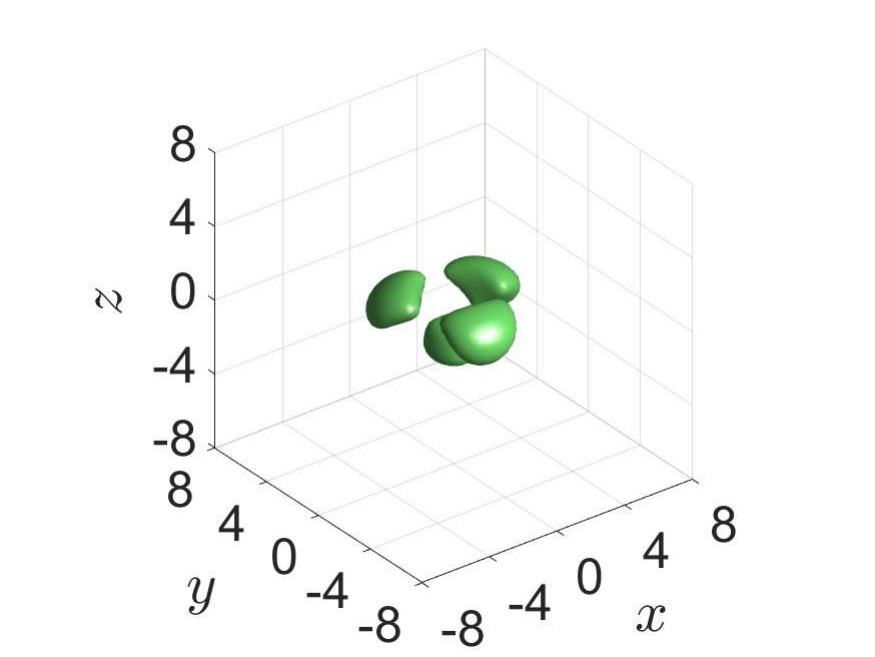}
 \\
  \centering
  \includegraphics[scale=0.35]{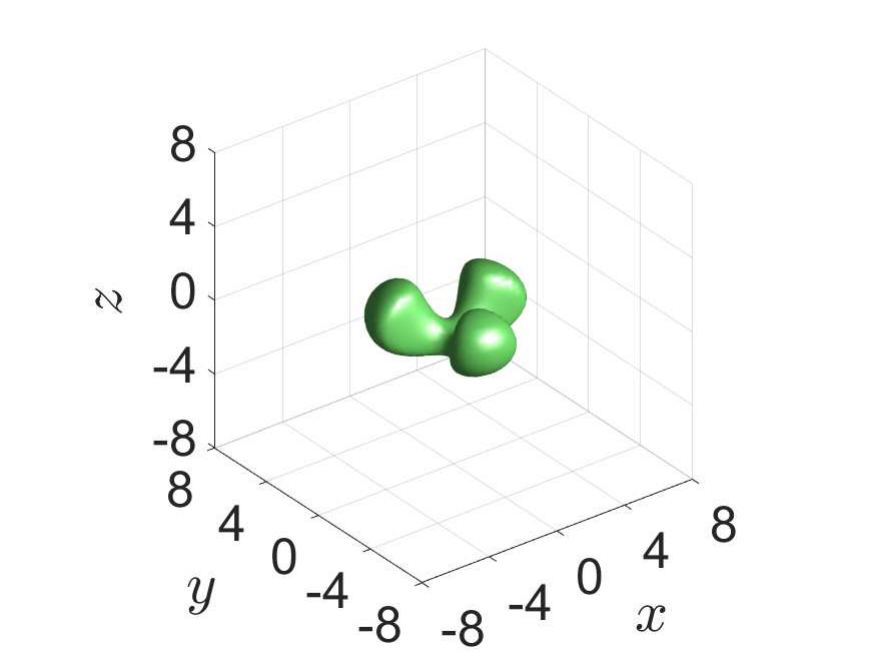}
  \hspace{-0.9cm}
  \includegraphics[scale=0.35]{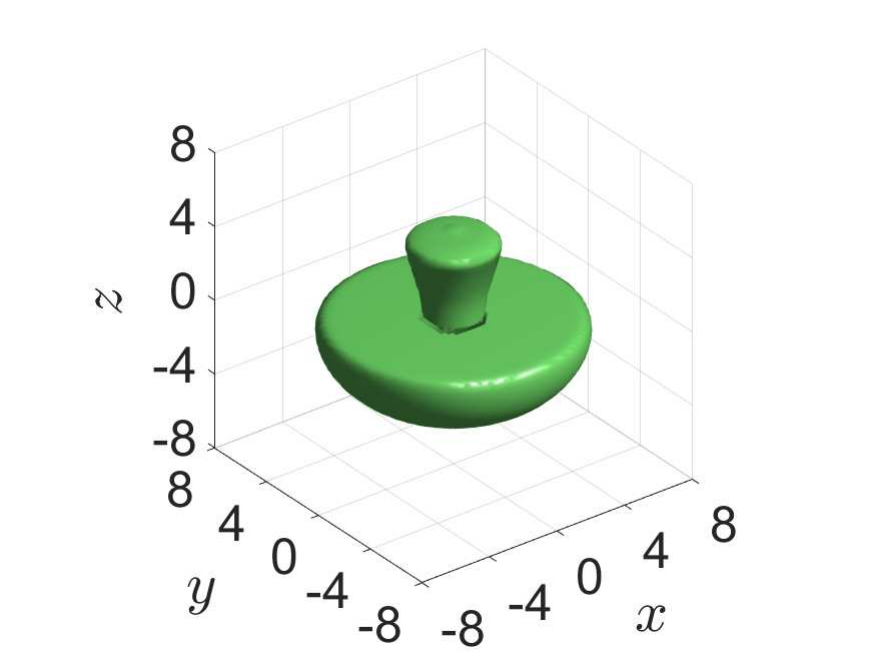}
  \hspace{-0.9cm}
  \includegraphics[scale=0.35]{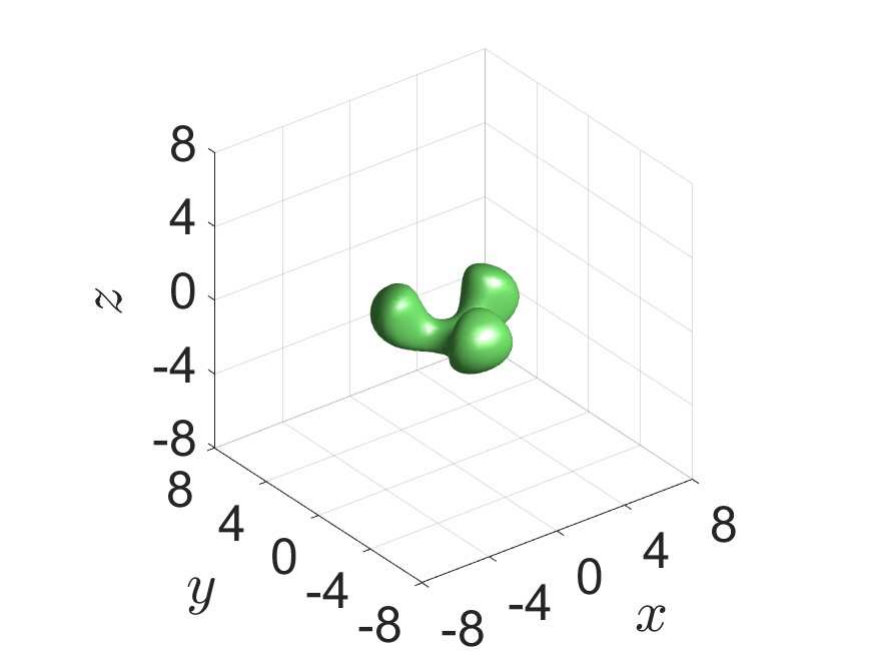}
\caption{
In Example \ref{amplitudes}, isosurface of the Bogoliubov amplitudes for {\bf Case III}: $u_{j}^{62}=10^{-3}$ (1st row), $v_j^{62}=10^{-12}$ (2nd row),
and {\bf Case IV}: $u_{j}^{42}=10^{-9}$ (3rd row), $v_j^{42}=10^{-10}$ (4th row) (from left to right: $j=1,0,-1$).}
\label{eigv_ani}
\end{figure}


\section{Conclusion}
\label{sec:Conclusion}

We proposed an efficient and spectrally accurate solver for
computing the BdG equations of the spin-1 BECs.
We first investigated its analytical eigenpairs, the structure of the generalized nullspace, and the bi-orthogonal property of eigenspaces on a continuous level.
Then, based on a stable Gram-Schmidt bi-orthogonal algorithm, making use of the generalized nullspace and Fourier spectral method, we develop a stable, efficient, and accurate eigensolver. The solver is matrix-free and the most time-consuming matrix-vector product can be accelerated with FFT achieving almost optimal complexity ($O({\tt DOF}\log ({\tt DOF}))$ operations).
Furthermore, we rigorously proved its numerical stability and spectral convergence.
Extensive numerical results confirm the spectral accuracy, efficiency in different dimensions,
and we investigated the excitation spectrum (eigenvalues) and Bogoliubov amplitudes (eigenfunctions) around the ground state with different parameters.
Extensions to BdG excitations of other BECs, such as multi-component BEC, spinor-dipolar BEC, etc, are feasible with minor adaptations on the specific matrix-vector product.

\vspace{-0.1cm}
\section*{Acknowledgements}
Y. Li was partially supported by the National Natural Science Foundations of China (No.12101447, 12271395).
M. Xie was partially supported by the National Natural Science Foundation of China (No. 12001402).
Z. Li and Y. Zhang were partially supported by the National Natural Science Foundation of China (No. 12271400).





\appendix
\section{Supplement to the convergence analysis}\label{appdix}
The coerciveness and conformal finite-dimension subspace approximation of the original BdG problem ( i.e., Eqn. \eqref{coer_BdG} and \eqref{infSupCondVecFun}) are essential properties to establish error estimates, therefore, without loss of generality, we shall consider the following equivalent but much simpler eigenvalue problem: to find $(\lambda,u)\in \mathbb{R}\times H^m_p(\Omega)$ such that
\bea\label{evp}
\mathcal L u:=-\Delta u + \gamma(\bx) u=\lambda\, u,
\eea
with constraint $\|u\|=1$ in one dimension space. Furthermore, we assume the operator $\mathcal L$ is symmetric positive semi-definite and $\gamma(\bx)$ is a smooth function. Extensions to high-dimensional (2D/3D) cases are straightforward and we shall omit them for brevity.

Define the following inner products:
\begin{align*}
&a(u,v) = \lag \nabla u, {\nabla v}\rag + \lag\gamma u, v\rag, \quad b(u,v) = \lag u, v \rag,\ \ \ \forall \, u,v \in H^m_p(\Omega)\\
&\lag u,v \rag = \int_\Omega u \bar v \dif {x}\ \ \mbox{and}\ \ \lag \nabla u ,{\nabla v}\rag = \int_\Omega  \nabla u \cdot\overline{\nabla v} \dif{x}.
\end{align*}
It is easy to prove that $a(\cdot,\cdot)$ is an inner product and coercive in $H^m_p(\Omega)$, that is,
\begin{equation*}
a(v,v) \leq C_a \|v\|_1 ^2,\ \ \forall\, v \in H^m_p(\Omega),
\end{equation*}
where $C_a$ is a constant.
In cases where $\ns(\mathcal L)$ is non-empty, we define $V := \ns(\mathcal L)^\perp = \{v\in H^m_p(\Omega)\, |\,  \lag v, w\rag = 0, \ w\in \ns(\mathcal L)\}$ as
the orthogonal complementary subspace, and the coerciveness ($V$-ellipticity) is satisfied as
\begin{equation}\label{a_coer}
c_a \|v\|_1 ^2 \leq a(v,v),\ \ \forall\, v \in V.
\end{equation}
In this appendix, we are concerned only with nonzero eigenvalues and their corresponding eigenfunctions, that is, solving the following weak form of Eqn. \eqref{evp}: to find $(\lambda,u) \in \mathbb{R}\times V$ such that
\bea\label{wf-evp}
a(u,v) = \lambda\, b(u,v), \ \ \ \forall\, v\in {V},
\eea
subject to constraint $b(u,u)=1$.

From Lemma \ref{FouApprox}, we could define the approximation finite dimensional spaces:
$$ {V}_N:= X_N\cap V\subset V.$$
The spectral-collocation method for \eqref{evp} amounts to finding $(\lambda_N,u_N) \in \mathbb{R}\times V_N$ such that
\begin{eqnarray}\label{discrte_form}
-\Delta u_N({x}_n) + \gamma({x}_n)~u_N({x}_n)  = \lambda_N\, u_N({x}_n), \quad  0\leq n \leq N-1,
\end{eqnarray}
with constraint $b(u_N,u_N)=1$ and $x_n = -L+nh_x \in \mathcal{T}_{\bx}$.
For convergence analysis, it is necessary to consider the following approximation problem: find $(\lambda_N,u_N) \in \mathbb{R}\times {V}_N$ such that
\bea\label{wf-a}
a_N(u_N,v_N) = \lambda_N\, b_N(u_N,v_N), \ \ \ \forall\, v_N\in {V}_N,
\eea
 subject to constraint  $b_N(u_N,u_N)=1$ with
\begin{align*}
&a_N(u_N,v_N) := \lag\nabla u_N, \nabla v_N\rag_N + \lag\gamma\, u_N, v_N \rag_N,\quad
b_N(u_N,v_N) := \lag u_N, v_N\rag_N,\\
&\lag u,v\rag_N := \frac{2L}{N}\sum\limits_{n=0}^{N-1}u_N({x}_n) \overline{v_N}({x}_n), \ \ \mbox{and}\  \lag \nabla u_N,\nabla v_N\rag_N := \frac{2L}{N}\sum\limits_{n=0}^{N-1}\nabla u_N({x}_n)\cdot \overline{\nabla v_N}({x}_n).
\end{align*}

\begin{lemma}\label{lem:equivalent}
The discrete problem \eqref{discrte_form} and discrete variational problem \eqref{wf-a} are equivalent.
\end{lemma}
\begin{proof}
 Multiplying both sides of \eqref{discrte_form} by $\overline{W_j}(x_\ell) = e^{-\im \mu_j (x_\ell+L)}$ and adding up $\ell$ from $1$ to $N-1$, we have
\beas
\lag-\Delta u_N, {W}_j\rag_N + \lag \gamma \, u_N, {W_j}\rag_N= \lambda_N \lag u_N,{W_j}\rag_N.
\eeas
Setting $u_N({x})= \sum\nolimits_{j=-N/2}^{N/2-1}~c_j W_j({x})$, by the orthogonality of basis $W_j(x)$, we have \beas
\lag\nabla u_N,\nabla W_j\rag_N = \sum\nolimits_{k=-N/2}^{N/2-1} c_k \mu_k\mu_j \lag W_k,W_j\rag_N
=  c_j \mu_j^2 = \lag -\Delta u_N,{W_j}\rag_N.
\eeas
The above identity holds true for the basis of $V_N$, therefore, the equivalence proof is completed.
\end{proof}

Before introducing the convergence results, we define the following notation
\begin{align*}
\delta_N(u)  & = \inf_{v\in V_N}\bigg\{\|u-v\|_1 + \sup_{w\in V_N}\frac{|a(v,w)-a_N(v,w)|}{\|w\|_1}\bigg\}.
\end{align*}

\begin{lemma}\label{lem_convergence}
Assume the $u({x})$ is a smooth and compactly supported function, i.e., $\supp\{u\} \subsetneq \Omega$, then the following estimate holds:
\begin{align}
\delta_N(u)\lesssim N^{-(m-\sigma)}\max\{|u|_m, |\gamma u|_m\},
\end{align}
where $\sigma = \max\{1,d/2\}$.
\end{lemma}
\begin{proof}
From Lemma \ref{FouApprox}, we have
\begin{align}
\|u-\cI u\|_1 \lesssim N^{-(m-1)}|u|_m.
\end{align}
From the definition of $X_N$, we can obtain following two equations,
\bea\label{cont_prop}
b_N(v_N,w_N) = b(v_N,w_N), \quad \lag \nabla v_N,\nabla w_N\rag_N = \lag \nabla v_N,\nabla w_N\rag, \quad \forall u_N,v_N\in V_N.
\eea
Based on \eqref{cont_prop} and \eqref{cont_prop}, we obtain
$$\big|a(v_N,w_N)-a_N(v_N,w_N)\big| = \big|\lag\gamma v_N,w_N\rag - \lag\gamma v_N,w_N\rag_N\big|,\quad \forall\, v_N, w_N\in V_N.$$
Hence, combing $\cI u\in V_N$, the following inequalities hold
\begin{align*}
&\big|a(\cI u,w_N)-a_N(\cI u,w_N)\big| = \big|\lag\gamma\, \cI u,w_N\rag - \lag\gamma\, \cI u,w_N\rag_N\big|\\
\leq & \big| \lag\gamma\, \cI u,w_N\rag - \lag\gamma\, u,w_N\rag\big|
+\big|\lag\gamma\, u,w_N\rag - \lag\gamma\, u,w_N\rag_N \big| +\big|\lag\gamma\, u,w_N\rag - \lag\gamma\, \cI u,w_N\rag_N \big|\\
:= & \ I_1+I_2+I_3.
\end{align*}
For the first term, combing Lemma \ref{FouApprox}, we have
\begin{align*}
I_1 & = \big| \lag\gamma\, \cI u,w_N\rag - \lag\gamma\, u,w_N\rag\big|\leq \|\gamma\|_\infty\|\widetilde u-u\|_0\|w_N\|_0 \lesssim N^{-m}|u|_m\|w_N\|_1.
\end{align*}
Since $\supp\{u\}\subsetneq \Omega$, $\supp\{\gamma u\}\subsetneq\Omega$, $\gamma u$ is smooth and periodic on $\Omega$. Setting $w_N({x}) = \sum_{k=-N/2}^{N/2-1}c_kW_k({x})$, and combining \cite[Theorem 2.3]{SpectralBkShen}, the following estimates hold
\begin{align*}
I_2 & = \big|\lag\gamma\, u,w_N\rag - \lag\gamma\, u,w_N\rag_N \big| \leq  \sum\limits_{k=-N/2}^{N/2-1}|c_k| \big|\lag\gamma\, u,W_k\rag - \lag\gamma\, u,W_k\rag_N\big|\\
&\leq \sum\limits_{k=-N/2}^{N/2-1} |c_k|\left|(\widehat{\gamma u})_k - (\widetilde{\gamma u})_k\right| \leq \sqrt{\sum_j\left|(\widehat{\gamma u})_j - (\widetilde{\gamma u})_j\right|^2}\sqrt{\sum_j |w_j|^2}\\
&\lesssim N^{-m}|\gamma u|_m\|w_N\|_0 \leq N^{-m}|\gamma u|_m\|w_N\|_1,
\end{align*}
where $(\widehat{\gamma u})_k = \int_\Omega \gamma({x})u({x})e^{-\im k{x}}\dif{x}$. For the third term, combing Lemma \ref{FouApprox}, we obtain
\begin{align*}
I_3 & = \big|\lag\gamma\, u,w_N\rag - \lag\gamma\, \cI u,w_N\rag_N \big| \leq \|\gamma\|_\infty\|u-\widetilde u\|_\infty\|w_N\|_0 \lesssim N^{-(m-d/2)}|u|_m\|w_N\|_1.
\end{align*}
Based on the definition of $\delta_N(u)$ and above error estimates, we have
\begin{align*}
\delta_N(u) &\lesssim N^{-(m-1)}|u|_m + N^{-m}|u|_m + N^{-m}|\gamma u|_m + N^{-(m-d/2)}|u|_m \\
& \lesssim N^{-(m-\sigma)}\max\{|u|_m, |\gamma u|_m\}.
\end{align*}

\end{proof}

\vspace{-0.3cm}
Based on  \eqref{a_coer} and the general theory of the error estimates for the eigenvalue problems by the approximation method \cite[Section 8]{Babuska_book} and \cite[Section 1]{SpectralBkShen}, we have the following convergence results. 

\begin{lemma}\label{lem:ErrEsti}
For any eigenpair approximation $(\lambda_N,u_N)$ of \eqref{discrte_form}, there is
an eigenpair  $(\lambda,u)$ of \eqref{wf-evp} corresponding to $\lambda$ such that
\beas
\|u-u_N\|_{1} \lesssim \delta_N(u),\quad  \|u-u_N\| \lesssim \zeta_N\|u-u_N\|_{1},\quad |\lambda-\lambda_N| \lesssim \|u-u_N\|_1^2,
\eeas
where $\zeta_N\rightarrow 0 \mbox{\ as\ }N\rightarrow \infty$.
\end{lemma}

\bibliographystyle{plain}

\end{document}